\documentclass{article}

\usepackage{arxiv}

\usepackage[utf8]{inputenc} 
\usepackage[T1]{fontenc}    
\usepackage{hyperref}       
\usepackage{url}            
\usepackage{booktabs}       
\usepackage{amsfonts}       
\usepackage{nicefrac}       
\usepackage{microtype}      
\usepackage{lipsum}		
\usepackage{graphicx}
\usepackage{mathtools}
\usepackage{amssymb}
\usepackage{amsthm}
\usepackage{amsmath}
\usepackage{doi}

\usepackage{cases}
\usepackage{microtype}
\usepackage[utf8]{inputenx}
\usepackage{amssymb}
\usepackage{amsthm}
\usepackage{mathtools, cases}
\usepackage{mathbbol}
\usepackage{makecell}
\usepackage{dsfont}
\usepackage{bbm}
\usepackage{mathabx}
\allowdisplaybreaks

\renewcommand{\sinh}{\operatorname{sinh}}
\renewcommand{\cosh}{\operatorname{cosh}}

\newtheorem{thm}{Theorem}

\title{Planar random motions in a vortex}

\author{ 
	\href{0000-0002-6421-533X}{Enzo Orsingher}\\
	Department of Statistical Sciences\\
	Sapienza University of Rome\\
	\texttt{enzo.orsingher@uniroma1.it} \\
	\And
	\href{https://orcid.org/0000-0002-6163-044X}{Manfred Marvin Marchione} \\
	Department of Statistical Sciences\\
	Sapienza University of Rome\\
	\texttt{manfredmarvin.marchione@uniroma1.it}}

\date{\today}

\renewcommand{\headeright}{}
\renewcommand{\undertitle}{}

\begin{document}
\maketitle

\begin{abstract}
We study a planar random motion $\big(X(t),Y(t)\big)$ with orthogonal directions which can turn clockwise, counterclockwise and reverse its direction each with a different probability. The support of the process is given by a time-varying square and the singular distributions on the boundary and the diagonals of the square are obtained. In the interior of the support, we study the hydrodynamic limit of the distribution. We then investigate the time $T(t)$ spent by the process moving vertically and the joint distribution of $\big(T(t),Y(t)\big)$. We prove that, in the hydrodynamic limit, the process $\big(X(t),Y(t)\big)$ spends half the time moving vertically.
\end{abstract}


\section{Introduction}
\noindent Finite-velocity planar random motions have been investigated by several authors in the literature since the Eighties. These motions can be classified according to different criteria. One such classification considers the number of possible directions the motion can take, distinguishing between minimal (Di Crescenzo \cite{dicrescenzocyclicplanar}), orthogonal (Orsingher \cite{orsingher2000}) and motions with infinite directions (Orsingher and De Gregorio \cite{degregorio}). A categorization is also possible based on the distribution of the waiting time before a change of direction occurs. While the classical problem involves direction changes at Poisson times, variants such as the Polya process (Crimaldi et al. \cite{polya}), processes with gamma-distributed intertimes (Martinucci et al. \cite{gamma}) and the geometric counting process (Di Crescenzo et al. \cite{geom}) have been explored in the univariate case. In the bivariate case, Cinque and Orsingher \cite{cinque} examined the case in which the direction changes are governed by a non-homogeneous Poisson process. Fractional extensions of finite-velocity planar random motions have been investigated by Masoliver and Lindenberg \cite{masoliver} and Masoliver \cite{masoliver2}.\\
\noindent A relevant criteria for classifying planar random motion considers the mechanism for determining the direction taken after a change occurs. Orsingher \cite{orsingher2000} investigated a motion with orthogonal directions which turns clockwise or counter-clockwise with equal probabilities. Di Crescenzo \cite{dicrescenzocyclicplanar} examined a minimal cyclic random motion, which can only turn counter-clockwise when the change of direction occurs. A cyclic planar random motion with orthogonal directions was examined by Orsingher et al. \cite{ogz}. Kolesnik and Orsingher \cite{kolesnik} and Cinque and Orsingher \cite{cinque} discussed a variant of the planar motion which can turn clockwise, counter-clockwise and can reverse its direction, each with an equal probability of $\frac{1}{3}$. The term \textit{reflection} was used by the authors to denote the direction reversal, and they showed that expressing the distribution of the process explicitly becomes a challenging task when reflection is present.\\
Multivariate extensions have been proposed, for minimal cyclic random motions, by Samoilenko \cite{samoilenko} and Lachal et al. \cite{lachaletal}. Lachal \cite{lachal} studied multidimensional random motions with an arbitrary number of directions, while three-dimensional random motions with orthogonal directions were investigated by Cinque and Orsingher \cite{cinque2023}.\\
\noindent In this paper, we study a finite-velocity planar random motion for which orthogonal changes of direction and reflection occur with different probabilities. We consider a planar random motion $\big(X(t),Y(t)\big)$ which moves along four directions
$$d_j=\Big(\cos\left(\frac{\pi}{2}j\right),\;\sin\left(\frac{\pi}{2}j\right)\Big),\qquad j=0,1,2,3.$$ Of course, we have that $d_j=d_{j+4n}$ for integer values of $n$. We assume that the random vector $\big(X(t),Y(t)\big)$ lies at the origin of the Cartesian plane at the initial time $t=0$, and it starts moving along one of the four possible directions with equal probabilities. The changes of direction are modeled as Poisson arrivals and we denote by $N(t),\;t>0$, the total number of changes. The intensity of $N(t)$ is assumed to be constant and we denote it by $\lambda$. Moreover, we assume that at each Poisson event the new direction is taken according to the following rule. With probability $p$ the direction changes by a counter-clockwise turn, therefore passing from $d_j$ to $d_{j+1}$. With probability $q$ the direction changes by a clockwise turn, therefore passing from $d_{j+1}$ to $d_{j}$. A reflection occurs wih probability $1-p-q$, changing the direction from $d_j$ to $d_{j+2}$. Thus, reflection is possible only if $p+q<1$. 
\noindent In the special case $p=q=\frac{1}{2}$, the process we study reduces to that studied by Kolesnik and Orsingher \cite{kolesnik} for which an explicit representation of the distribution can be obtained. If $p\neq q$, a vorticity effect is introduced in the sense that the process tends to perform clockwise turns with higher probability if $p<q$, while counter-clockwise turns are more likely to occur if $p>q$. In the limiting cases $p=1$ or $q=1$, the cyclic random motion investigated by Orsingher et al. \cite{ogz} is obtained. It is clear that the support of the motion is time-dependent and coincides with the square $S_{ct}$ defined as $$S_{ct}=\left\{(x,y)\in\mathbb{R}^2:\;|x+y|\le ct,\;|x-y|\le ct\right\}.$$ Moreover, the distribution of $\big(X(t),Y(t)\big)$ has a singular component on the boundary $\partial S_{ct}$ of the domain $S_{ct}$. The particle with position $\big(X(t),Y(t)\big)$ lies on $\partial S_{ct}$ if no changes of direction occur or if clockwise and counter-clockwise turns alternate. In particular, if no changes of direction occur the particle lies on the vertices of the square $S_{ct}$. A sketch of some sample paths is provided in figure \ref{fig:samplepaths}.\\
\begin{figure}[h]
\centering
\includegraphics[scale=0.75]{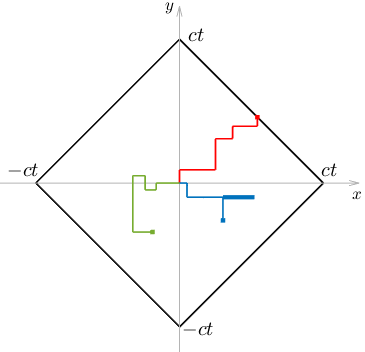}
\caption{some sample paths of the process $\big(X(t),Y(t)\big)$. The motion takes place in the time-varying square $S_{ct}$ and a degenerate component of the distribution is present on the boundary and the diagonals of the square. In particular, the process lies on $\partial S_{ct}$ if the path changes direction by always alternating two contiguous directions. This behaviour is displayed, for instance, by the red path in the figure. The green and blue sample paths lead the particle to the interior of the support. While the green path only exhibits clockwise and counterclockwise changes of direction, in the blue path a reflection has occurred and the particle has reversed its direction.}\label{fig:samplepaths}
\end{figure}

Our work starts with the presentation of some results concerning the distribution of the process $\big(X(t),Y(t)\big)$ on the boundary $\partial S_{ct}$. In the general case where reflection is admitted, the probability of the particle lying on the boundary is given by
$$\mathbb{P}\Big(\big(X(t),Y(t)\big)\in \partial S_{ct}\Big)=e^{-\lambda t}\left\{e^{\lambda t\sqrt{pq}}\left(1+\frac{p+q}{2\sqrt{pq}}\right)+e^{-\lambda t\sqrt{pq}}\left(1-\frac{p+q}{2\sqrt{pq}}\right)-1\right\}.$$
On each side of $\partial S_{ct}$, the distribution of $\big(X(t),Y(t)\big)$ admits a continuous component which is explicitly obtained in this paper. In particular, we prove that 
\begin{align}\mathbb{P}\Big(X(t)+&Y(t)=ct,\; X(t)-Y(t)\in \mathop{d\eta}\Big)/\mathop{d\eta}\nonumber\\
&=\frac{e^{-\lambda t}}{4c}\left[\frac{\lambda(p+q)}{2}\;I_0\left(\frac{\lambda}{c}\sqrt{pq}\sqrt{c^2t^2-\eta^2}\right)+\;\frac{\partial}{\partial t}I_0\left(\frac{\lambda}{c}\sqrt{pq}\sqrt{c^2t^2-\eta^2}\right)\right],\qquad\lvert\eta\lvert<ct\nonumber\end{align}
with characeristic function
\begin{align}\mathbb{E}\left[e^{i\alpha\left(X(t)-Y(t)\right)}\;\mathds{1}_{\{X(t)+Y(t)=ct\}}\right]=\frac{e^{-\lambda t}}{4}&\left[\left(1+\frac{\lambda(p+q)}{2\sqrt{\lambda^2pq-\alpha^2c^2}}\right)e^{t\sqrt{\lambda^2pq-\alpha^2c^2}}\right.\nonumber\\+&\left.\left(1-\frac{\lambda(p+q)}{2\sqrt{\lambda^2pq-\alpha^2c^2}}\right)e^{-t\sqrt{\lambda^2pq-\alpha^2c^2}}\right].\nonumber\end{align}
We then proceed with our analysis by studying the distribution of $\big(X(t),Y(t)\big)$ on the diagonals of $S_{ct}$, that is on the set $$Q_{ct}=\left\{(x,y)\in S_{ct}:\;x=0\vee y=0\right\}.$$ We first show that 
\begin{equation}\label{intro:PQ}\mathbb{P}\Big(\big(X(t),Y(t)\big)\in Q_{ct}\Big)=e^{-\lambda (p+q)t}.\end{equation}
Formula (\ref{intro:PQ}) above highlights that if $p+q=1$, that is if reflection is not admitted, the probability of the particle lying on the diagonals of the square $S_{ct}$ coincides with the probability of no Poisson events occurring, which implies that no direction changes occur and therefore the particle must lie on one of the vertices of $S_{ct}$. In other words, if reflection is not admitted the particle lies on $Q_{ct}$ if and only if it lies on one of the vertices of $S_{ct}$. Conversely, if reflection is possible, the distribution of $\big(X(t),Y(t)\big)$ on $Q_{ct}$ admits a continuous component. In fact, we show that, if $p+q<1$,
\begin{align}\mathbb{P}\Big( X(t)\in\mathop{dx},\;Y(t)=0\Big)/\mathop{dx}=\frac{e^{-\lambda t}}{4c}\Bigg[\lambda(1-p-q)&\; I_0\left(\frac{\lambda}{c}(1-p-q)\sqrt{c^2t^2-x^2}\right)\nonumber\\+&\frac{\partial}{\partial t}I_0\left(\frac{\lambda}{c}(1-p-q)\sqrt{c^2t^2-x^2}\right)\Bigg],\qquad \lvert x\lvert<ct\nonumber\end{align} and we obtain the corresponding characteristic function.\\
We then study the distribution of $\big(X(t),Y(t)\big)$ in the interior of its support $S_{ct}$. We show that the probability density function \begin{equation}\label{intro:u}u(x,y,t)\mathop{dx}\mathop{dy}=\mathbb{P}\Big(X(t)\in\mathop{dx},\;Y(t)\in\mathop{dy}\Big),\qquad \lvert x\lvert+\lvert y\lvert< ct\end{equation} is the solution to a complicated fourth-order partial differential equation for which a general solution is difficult to give in an explicit form. In the special case in which reflection is not admitted, we derive the characteristic function of $\big(X(t),Y(t)\big)$ in the interior of $S_{ct}$ explicitly. In the general case, we are able to study the hydrodynamic limit of the distribution for $\lambda,c\to+\infty$, with $\frac{\lambda}{c^2}\to1$, and we prove that the heat equation
\begin{equation}\label{intro:heat}\frac{\partial u}{\partial t}=\frac{1}{4}\;\frac{(1-p)+(1-q)}{(1-p)^2+(1-q)^2}\;\Delta u\end{equation} is satisfied in the limiting case. The interpretation of equation (\ref{intro:heat}) is that the process $\big(X(t),Y(t)\big)$ converges to a planar Brownian motion with independent components and diffusion coefficient depending on $p$ and $q$. It is interesting to observe that the diffusion coefficient is maximized for $p=q=\frac{1}{2}$.\\

The final part of our work is devoted to the study of the time spent by $\big(X(t),Y(t)\big)$ moving vertically, in parallel to the $y$-axis. Thus, we study the random process \begin{equation*}\label{Tdef}T(t)=\int_0^t\mathds{1}_{\{D(\tau)\in\{d_1,d_3\}\}}\mathop{d\tau},\qquad t>0.\end{equation*} 
We obtain the exact distribution of $T(t)$ and its characteristic function. In particular, we prove that
\begin{align}\mathbb{P}\Big(T(t)\in\mathop{ds}\Big)/ds=e^{-\lambda(p+q) t}\Bigg[\lambda(p+&q)\; I_0\left(2\lambda(p+q)\sqrt{s(t-s)}\right)\nonumber\\
&+\frac{\partial}{\partial t}I_0\left(2\lambda(p+q)\sqrt{s(t-s)}\right)\Bigg],\qquad s\in(0,t).\end{align}

We then analyze the joint distribution of the vector $\big(T(t),Y(t)\big)$. The support of such random vector is given by the triangle $$R_{ct}=\left\{(s,y)\in\mathbb{R}^2:\;\lvert y\lvert\le cs,\,s\in[0,t]\right\}.$$ The process $\big(T(t),Y(t)\big)$ can move along two directions which are parallel to the oblique sides of $R_{ct}$. Moreover, as we will discuss in the remainder of the paper, the process can arrest its movement at some points and stay still for some time. In particular, this occurs when the process $\big(X(t),Y(t)\big)$ is moving horizontally in $S_{ct}$. Some sample paths of the process $\big(T(t),Y(t)\big)$ are exemplified in figure \ref{fig:Rsamplepaths}.

\begin{figure}[h]
\centering
\includegraphics[scale=0.75]{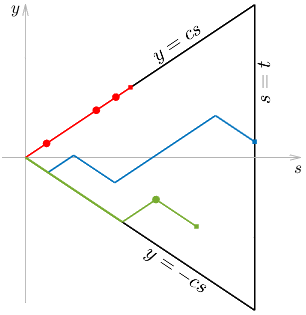}
\caption{some sample paths of the process $\big(T(t),Y(t)\big)$, whose support coincides with the time-varying triangle $R_{ct}$. The particle with position $\big(T(t),Y(t)\big)$ can move along two diffrent directions, parallel to the oblique side of the triangle, and can arrest its movement at some points and stay still for some time. For each path in the figure, we highlighted with a dot the points in which the particle has stopped. In particular, the particle which moved along the blue path has never arrested its movement and therefore it lies on the verical side of $R_{ct}$. The particle lies on one of the oblique sides of $R_{ct}$ if it always moves along the same direction, possibly stopping its movement sometimes. This is the case of the red path in the figure.}\label{fig:Rsamplepaths}
\end{figure}
\noindent We first study the distribution of $\big(T(t),Y(t)\big)$ in the interior of $R_{ct}$ and we show that the probability density function \begin{equation}\label{intro:ubar}\bar{u}(s,y,t)\mathop{ds}\mathop{dy}=\mathbb{P}\Big(T(t)\in \mathop{ds},\;Y(t)\in\mathop{dy}\Big),\qquad (s,y)\in R_{ct}.\end{equation} satisfies a third-order partial differential equation. By studying the hydrodynamic limit of the equation, we prove that, for $\lambda,c\to+\infty$ with $\frac{\lambda}{c^2}\to1$, the partial differential equation 
\begin{equation}\frac{\partial u}{\partial t}=-\frac{1}{2}\frac{\partial u}{\partial s}+\frac{1}{4}\;\frac{(1-p)+(1-q)}{(1-p)^2+(1-q)^2}\frac{\partial^2 u}{\partial y^2.}\label{intro:bivdegenerate}\end{equation} holds. The interpretation of equation (\ref{intro:bivdegenerate}) is that, by taking the hydrodynamic limit, the asymptotic distribution is $$\lim_{\lambda,c\to+\infty}\mathbb{P}\Big(T(t)\in \mathop{ds},\;Y(t)\in\mathop{dy}\Big)/\left(\mathop{ds}\mathop{dy}\right)=\sqrt{\frac{(1-p)^2+(1-q)^2}{(1-p)+(1-q)}}\cdot \frac{e^{-\frac{(1-p)^2+(1-q)^2}{(1-p)+(1-q)}\,\frac{y^2}{t}}}{\sqrt{\pi t}}\cdot \delta\left(s-\frac{t}{2}\right)$$ where the symbol $\delta(\cdot)$ denotes the Dirac delta. Therefore, while in the hydrodynamic limit the variable $Y(t)$ becomes a Brownian motion with variance depending on $p$ and $q$, the process $T(t)$ becomes deterministic and its limiting value is equal to $\frac{t}{2}$. Thus, in the limit, the process $\big(X(t),Y(t)\big)$ spends half of the time moving vertically.\\
The joint distribution of $\big(T(t),Y(t)\big)$ is also studied on the sides of the triangle $R_{ct}$. As for the oblique side, on which the particle lies if $Y(t)=c T(t)$, the distribution satisfies a third-order partial differential equation. In the special case $p+q=1$, we are able to study such equation in detail. We first calculate the probability of the particle $\big(T(t),Y(t)\big)$ lying on the oblique side of the triangle, which reads \begin{align*}\mathbb{P}\big(Y(t)=c\,T(t)\big)=\frac{e^{-\lambda t}}{8}\Bigg\{&\left(1+\frac{1}{\sqrt{2p(1-p)}}\right)^2 e^{\lambda t\sqrt{2p(1-p)}}\nonumber\\
&+\left(1-\frac{1}{\sqrt{2p(1-p)}}\right)^2 e^{-\lambda t\sqrt{2p(1-p)}}-\frac{(2p-1)^2}{p(1-p)}\Bigg\}.\nonumber\end{align*} We then investigate the exact distribution of $\big(T(t),Y(t)\big)$ when $Y(t)=cT(t)$. If $p+q=1$, we show that 
\begin{align*}\mathbb{P}\Big(T(t)\in\mathop{ds},\,Y(t)=c\,T(t)\Big)/\mathop{ds}=&\frac{\lambda}{2\sqrt{2}}\;I_0\Big(2\lambda\sqrt{2p(1-p)}\,\sqrt{s(t-s)}\Big)\nonumber\\&+\frac{1}{4}\left(1+\frac{1}{2p(1-p)}\right)\frac{\partial}{\partial t}I_0\Big(2\lambda\sqrt{2p(1-p)}\,\sqrt{s(t-s)}\Big),\qquad\lvert s\lvert<t.\nonumber\end{align*}
We then discuss a special case in which the partial differential equation governing the distribution of $\big(T(t),Y(t)\big)$ on the oblique side of $R_{ct}$ reduces to a second-order equation, namely the case $p=q$ in which clockwise and counterclockwise changes of direction are equiprobable. While we are not able to find the exact distribution, for $p=q$ we obtain the probability of the particle lying on the oblique side of $R_{ct}$, that is
\begin{align*}\mathbb{P}\big(Y(t)=c\,T(t)\big)=&\frac{3}{8}\left(1+\frac{(1+6p)}{3\sqrt{12p^2-4p+1}}\right)\;e^{\frac{\lambda t}{2}\left(-(1+2p)+\sqrt{12p^2-4p+1}\right)}\nonumber\\&\;\;+\frac{3}{8}\left(1-\frac{(1+6p)}{3\sqrt{12p^2-4p+1}}\right)\;e^{\frac{\lambda t}{2}\left(-(1+2p)-\sqrt{12p^2-4p+1}\right)}.\end{align*}
Finally, we study the distribution of $\big(T(t),Y(t)\big)$ on the vertical side of $R_{ct}$ and we show that it perfecly resembles the distribution of the process $\big(X(t),Y(t)\big)$ on the diagonals $Q_{ct}$ of the square $S_{ct}$. Thus, it holds that \begin{align*}\mathbb{P}\Big(Y(t)\in\mathop{dy},\,T(t)=t\Big)/\mathop{dy}=\frac{e^{-\lambda t}}{4c}\Bigg[\lambda&(1-p-q)\; I_0\left(\frac{\lambda}{c}\,(1-p-q)\,\sqrt{c^2t^2-y^2}\right)\nonumber\\
&+\frac{\partial}{\partial t}I_0\left(\frac{\lambda}{c}\,(1-p-q)\,\sqrt{c^2t^2-y^2}\right)\Bigg],\qquad \lvert y\lvert<ct\end{align*} The paper is structured in the following manner. In section 2 we study the distribution of the process $\big(X(t),Y(t)\big)$ on the boundary $\partial S_{ct}$ of its support, while in section 3 we study the distribution on the diagonals $Q_{ct}$. Section 4 is devoted to the study of the distribution in the interior of the square $S_{ct}$. In the final section, the distribution of the time spent moving vertically $T(t)$ and the joint distribution of the vector $\big(T(t),Y(t)\big)$ are studied.

\section{Distribution on the boundary}\label{sec:boundary}
\noindent In this section, we study the behaviour of the particle $\big(X(t),Y(t)\big)$ on the boundary of the square $S_{ct}$. We first calculate the probability of the particle lying on $\partial S_{ct}$ and we then obtain the exact distribution of the position of the particle on the boundary.  We examine the general case in which reflection is possible. Our analysis starts by observing that, regardless of the initial direction of the motion, the boundary can be reached if the particle always moves along two contiguous directions by alternating clockwise and counterclockwise changes of direction. Of course, if no changes of direction occur, the motion always takes place on $\partial S_{ct}$ because the particle lies on the vertices of the square. Therefore, we distinguish three cases in which the particle is on $\partial S_{ct}$:
\begin{itemize}
\item the particle is on the boundary because no changes of direction occur
\item an even number of alternating changes of direction occurs, in which case the first change of direction can be either clockwise or counterclockwise
\item an odd number of alternating changes of direction occurs, in which case the first change of direction can be either clockwise or counterclockwise.
\end{itemize}
Thus, we have the following result.
\begin{thm}It holds that
\begin{align}\mathbb{P}\Big(\big(X(t),Y(t)\big)\in \partial S_{ct}\Big)=e^{-\lambda t}\left\{e^{\lambda t\sqrt{pq}}\left(1+\frac{p+q}{2\sqrt{pq}}\right)+e^{-\lambda t\sqrt{pq}}\left(1-\frac{p+q}{2\sqrt{pq}}\right)-1\right\}.\label{boundaryprob}\end{align}\end{thm}
\begin{proof}By treating separately the cases in which the process performs no changes of direction, an even number of changes and an odd number of changes, we can write that
\begin{align}\mathbb{P}\Big(\big(X&(t),Y(t)\big)\in \partial S_{ct}\Big)\nonumber\\
=&\mathbb{P}\big(N(t)=0\big)+\sum_{k=1}^{\infty}\mathbb{P}\big(N(t)=2k\big)\left(p^kq^k+q^kp^k\right)+\sum_{k=0}^{\infty}\mathbb{P}\big(N(t)=2k+1\big)\left(p^{k+1}q^k+q^{k+1}p^k\right)\nonumber\\
=&e^{-\lambda t}+2e^{-\lambda t}\sum_{k=1}^{\infty}\frac{(\lambda t\sqrt{pq})^{2k}}{(2k)!}+e^{-\lambda t}\frac{(p+q)}{\sqrt{pq}}\sum_{k=0}^{\infty}\frac{(\lambda t\sqrt{pq})^{2k+1}}{(2k+1)!}\nonumber\\
=&e^{-\lambda t}+2e^{-\lambda t}\big(\cosh\left(\lambda t \sqrt{pq}\right)-1\big)+e^{-\lambda t}\frac{(p+q)}{\sqrt{pq}}\sinh\left(\lambda t \sqrt{pq}\right)\nonumber\\
=&e^{-\lambda t}\left\{e^{\lambda t\sqrt{pq}}\left(1+\frac{p+q}{2\sqrt{pq}}\right)+e^{-\lambda t\sqrt{pq}}\left(1-\frac{p+q}{2\sqrt{pq}}\right)\right\}-e^{-\lambda t}\nonumber\end{align} which completes the proof.
\end{proof}
We emphasize that formula (\ref{boundaryprob}) represents the probability of the particle being at any point on the entire boundary of the support, encompassing all four sides and the vertices of the square. In order to calculate the probability of the particle lying in the interior of a specific side of $\partial S_{ct}$, the probability mass on the vertices must be subtracted. The result must then be divided by four in order to focus on just one side. For instance, the probability of the process $\big(X(t),Y(t)\big)$ belonging to the interior of the side of $\partial S_{ct}$ which belongs to the first quadrant of the Cartesian plane is given by \begin{align}\mathbb{P}\Big(X(t)+Y(t)=c&t,\;\left\lvert X(t)-Y(t)\right\lvert<ct\Big)=\frac{1}{4}\left[\mathbb{P}\Big(\big(X(t),Y(t)\big)\in \partial S_{ct}\Big)-\mathbb{P}\big(N(t)=0\big)\right]\nonumber\\
=&\frac{e^{-\lambda t}}{4}\left\{e^{\lambda t\sqrt{pq}}\left(1+\frac{p+q}{2\sqrt{pq}}\right)+e^{-\lambda t\sqrt{pq}}\left(1-\frac{p+q}{2\sqrt{pq}}\right)-2\right\}.\label{boundaryprobside}\end{align}
We now highlight some interesting special cases of formula (\ref{boundaryprob}). First of all, if reflection is not admitted, that is if $p+q=1$, we have that
\begin{align}\mathbb{P}\Big(\big(X&(t),Y(t)\big)\in \partial S_{ct}\Big)\nonumber\\
\;&\;\;=e^{-\lambda t}\left\{e^{\lambda t\sqrt{p(1-p)}}\left(1+\frac{1}{2\sqrt{p(1-p)}}\right)+e^{-\lambda t\sqrt{p(1-p)}}\left(1-\frac{1}{2\sqrt{p(1-p)}}\right)-1\right\}.\label{trenitaliaprima}\end{align}
Moreover, a noticeable simplification of formula (\ref{boundaryprob}) occurs for $p=q$, in which case we can write
\begin{align}\label{trenitalia}\mathbb{P}\Big(\big(X(t),Y(t)\big)\in \partial S_{ct}\Big)=2e^{-\lambda t(1-p)}-e^{-\lambda t}.\end{align} For $p=q=\frac{1}{3}$, formula (\ref{trenitalia}) is consistent with the results obtained by Cinque and Orsingher \cite{cinque}.
In the special case $p=\frac{1}{2}$, both formulas (\ref{trenitaliaprima}) and (\ref{trenitalia}) reduce to $$\mathbb{P}\Big(\big(X(t),Y(t)\big)\in \partial S_{ct}\Big)=2e^{-\frac{\lambda t}{2}}-e^{-\lambda t}$$ which coincides with the result obtained by Orsingher \cite{orsingher2000}.\\

We now study the distribution of $\big(X(t),Y(t)\big)$ on a side of $\partial S_{ct}$. In particular, we consider the side belonging to the first quadrant of the Cartesian plane. Therefore, we want to determine the probability density function
\begin{equation}f(\eta,t)=\mathbb{P}\Big(X(t)+Y(t)=ct,\; X(t)-Y(t)\in \mathop{d\eta}\Big)/\mathop{d\eta},\qquad -ct<\eta< ct.\label{boundarydistr}\end{equation}
\noindent We start by considering, for $j=0,1$, the density functions
$$f_j(\eta,t)=\mathbb{P}\Big(X(t)+Y(t)=ct,\; X(t)-Y(t)\in d\eta,\;D(t)=d_0\Big)/\mathop{d\eta}.$$

\noindent Clearly, we have that \begin{equation}f(\eta,t)=f_0(\eta,t)+f_1(\eta,t).\label{boundarysum}\end{equation}
\noindent By means of standard methods it can be shown that
\begin{equation}\label{boundaryeqs}
\begin{dcases}
f_0(\eta,t+\mathop{dt})=f_0(\eta-c\mathop{dt},t)(1-\lambda \mathop{dt})+f_1(\eta,t)\lambda q\mathop{dt}+o(\mathop{dt})\\
f_1(\eta,t+\mathop{dt})=f_1(\eta+c\mathop{dt},t)(1-\lambda \mathop{dt})+f_0(\eta,t)\lambda p\mathop{dt}+o(\mathop{dt})
\end{dcases}\end{equation}
\noindent Therefore, by performing a first-order Taylor expansion of the equations in formula (\ref{boundaryeqs}), we obtain that $f_0$ and $f_1$ satisfy the system of partial differential equations
\begin{equation}\label{boundarydiffeqs}
\begin{dcases}
\frac{\partial f_0}{\partial t}=-c\frac{\partial f_0}{\partial \eta}+\lambda q f_1-\lambda f_0\\
\frac{\partial f_1}{\partial t}=c\frac{\partial f_1}{\partial \eta}+\lambda p f_0-\lambda f_1
\end{dcases}\end{equation}
\noindent with initial condition $f_0(\eta,0)=f_1(\eta,0)=\frac{1}{4}\,\delta(\eta).$ The relationship (\ref{boundarysum}) together with the system of differential equations (\ref{boundarydiffeqs}) implies that the probability density function $f$ satisfies, for $-ct<\eta<ct$, the partial differential equation
\begin{equation}\label{boundarypde}\left(\frac{\partial^2}{\partial t^2}+2\lambda \frac{\partial}{\partial t}-c^2\frac{\partial^2}{\partial \eta^2}+\lambda^2\left(1-pq\right)\right)f=0\end{equation}
\noindent Therefore, the continuous component of the distribution (\ref{boundarydistr}) is given in the following theorem.
\begin{thm}\label{thm:f}The probability density function (\ref{boundarydistr}) is
\begin{equation}f(\eta,t)=\frac{e^{-\lambda t}}{4c}\left[\frac{\lambda(p+q)}{2}\;I_0\left(\frac{\lambda}{c}\sqrt{pq}\sqrt{c^2t^2-\eta^2}\right)+\;\frac{\partial}{\partial t}I_0\left(\frac{\lambda}{c}\sqrt{pq}\sqrt{c^2t^2-\eta^2}\right)\right],\qquad\lvert\eta\lvert<ct.\label{boundarydistrthm}\end{equation}\end{thm}
\begin{proof}In view of equation (\ref{boundarypde}), the function \begin{equation}\widetilde{f}(\eta,t)=e^{\lambda t}f(\eta,t)\label{fcheck}\end{equation}
satisfies the partial differential equation
\begin{equation}\label{boundarypdecheck}\frac{\partial^2\widetilde{f}}{\partial t^2}-c^2\frac{\partial^2\widetilde{f}}{\partial \eta^2}=\lambda^2pq\widetilde{f}.\end{equation}
The change of variables \begin{equation}z=\sqrt{c^2t^2-\eta^2}\label{zvar}\end{equation} reduces equation (\ref{boundarypdecheck}) to the Bessel equation
\begin{equation*}\label{besself}\frac{d^2\widetilde{f}}{z^2}+\frac{1}{z}\frac{d\widetilde{f}}{dz}-\frac{\lambda^2}{c^2}pq\widetilde{f}=0\end{equation*}
whose general solution reads \begin{equation}\label{basselfsol}\widetilde{f}(z)=A\;I_0\left(\frac{\lambda}{c}\sqrt{pq}z\right)+B\;K_0\left(\frac{\lambda}{c}\sqrt{pq}z\right).\end{equation}
We disregard the term involving the modified Bessel function of the second kind $K_0(\cdot)$ because it would make the density function $f(\eta,t)$ non-integrable in proximity of the endpoints $\eta=\pm ct$. Therefore, by inverting the transformations (\ref{fcheck}) and (\ref{zvar}), we express the solution to equation (\ref{boundarypde}) in the form 
\begin{equation}\label{boundarygensol}f(\eta,t)=e^{-\lambda t}\left[A\; I_0\left(\frac{\lambda}{c}\sqrt{pq}\sqrt{c^2t^2-\eta^2}\right)+B\;\frac{\partial}{\partial t}I_0\left(\frac{\lambda}{c}\sqrt{pq}\sqrt{c^2t^2-\eta^2}\right)\right]\end{equation}
Comparing formulas (\ref{basselfsol}) and (\ref{boundarygensol}), observe that we added a term involving the time derivative of the Bessel function in order to add flexibility to the general solution to equation (\ref{boundarypde}). The introduction of the additional term can be performed because equation (\ref{boundarypde}) is homogeneous with respect to the time variable. We now have to determine the coefficients $A$ and $B$ in the expression (\ref{boundarygensol}). For this purpose, we use the well-known relation
\begin{equation}\int_{-ct}^{ct}I_0\left(K\;\sqrt{c^2t^2-\eta^2}\right)\mathop{d\eta}=\frac{1}{K}\left(e^{Kct}-e^{-Kct}\right),\qquad c,t>0\label{IntI}\end{equation}
from which we also obtain
\begin{equation}\label{IntdI}\int_{-ct}^{ct}\frac{\partial}{\partial t}I_0\left(K\;\sqrt{c^2t^2-\eta^2}\right)\mathop{d\eta}=c\left(e^{Kct}+e^{-Kct}-2\right),\qquad c,t>0.\end{equation}
Using the above formulas and integrating the expression (\ref{boundarygensol}), we obtain that
\begin{equation}\label{boundarygensolint}\int_{-ct}^{ct}f(\eta,t)\mathop{d\eta}=e^{-\lambda t}\left\{\left(Bc+\frac{Ac}{\lambda\sqrt{pq}}\right)e^{\lambda t\sqrt{pq}}+\left(Bc-\frac{Ac}{\lambda\sqrt{pq}}\right)e^{-\lambda t\sqrt{pq}}-2Bc\right\}.\end{equation}
Comparing formulas (\ref{boundaryprobside}) and (\ref{boundarygensolint}) yields
$$A=\frac{\lambda(p+q)}{8c},\qquad B=\frac{1}{4c}.$$
which completes the proof.
\end{proof}

In the following theorem, we derive an explicit expression for the characteristic function of $\big(X(t),Y(t)\big)$ on the boundary of the support.
\begin{thm}The characteristic function of $\big(X(t),Y(t)\big)$ on $\partial S_{ct}$ is
\begin{align}\label{boundarydistrchf}\mathbb{E}\left[e^{i\alpha\left(X(t)-Y(t)\right)}\;\mathds{1}_{\{X(t)+Y(t)=ct\}}\right]=\frac{e^{-\lambda t}}{4}&\left[\left(1+\frac{\lambda(p+q)}{2\sqrt{\lambda^2pq-\alpha^2c^2}}\right)e^{t\sqrt{\lambda^2pq-\alpha^2c^2}}\right.\nonumber\\+&\left.\left(1-\frac{\lambda(p+q)}{2\sqrt{\lambda^2pq-\alpha^2c^2}}\right)e^{-t\sqrt{\lambda^2pq-\alpha^2c^2}}\right]\end{align}
\end{thm}
\begin{proof}
By using the notation $$\widehat{f}(\alpha,t)=\mathbb{E}\left[e^{i\alpha\left(X(t)-Y(t)\right)}\;\mathds{1}_{\{X(t)+Y(t)=ct\}}\right]$$
equation (\ref{boundarypde}) implies that the following ordinary differential equation is satisfied
\begin{equation}\label{boundarychfode}
\begin{dcases}
\frac{d^2\widehat{f}}{dt^2}+2\lambda\frac{d\widehat{f}}{dt}+\left[c^2\alpha^2+\lambda^2(1-pq)\right]\widehat{f}=0\\
\widehat{f}(\alpha,0)=\frac{1}{2}\\
\frac{d}{d t}\left.\widehat{f}(\alpha,t)\right\lvert_{t=0}=-\frac{\lambda(2-p-q)}{4}
\end{dcases}\end{equation}
where the initial conditions have been determined from the linear system (\ref{boundarydiffeqs}). The general solution to equation (\ref{boundarychfode}) reads
$$\widehat{f}(\alpha,t)=k_0(\alpha)\;e^{-\lambda t+t\sqrt{\lambda^2pq-\alpha^2c^2}}+k_1(\alpha)\;e^{-\lambda t-t\sqrt{\lambda^2pq-\alpha^2c^2}}.$$
\noindent The coefficients $k_0$ and $k_1$ can be determined by using the initial conditions.
\end{proof}
We conclude this section by verifying that the results presented so far are consistent with each other. We start by observing that, for $\alpha=0$, formula (\ref{boundarydistrchf}) reduces to
\begin{align*}\mathbb{P}\Big(X(t)+Y(t)=ct\Big)=\frac{e^{-\lambda t}}{4}\left\{e^{\lambda t\sqrt{pq}}\left(1+\frac{p+q}{2\sqrt{pq}}\right)+e^{-\lambda t\sqrt{pq}}\left(1-\frac{p+q}{2\sqrt{pq}}\right)\right\}\end{align*} which is consistent with formula (\ref{boundaryprobside}). Moreover, a direct calculation of the characteristic function (\ref{boundarydistrchf}) can be performed by observing that
\begin{align}\mathbb{E}&\Big[e^{i\alpha\left(X(t)-Y(t)\right)}\;\mathds{1}_{\{X(t)+Y(t)=ct\}}\Big]\nonumber\\
&\qquad=e^{i\alpha ct}\cdot\mathbb{P}\Big(X(t)=ct,\,Y(t)=0\Big)+e^{-i\alpha ct}\cdot\mathbb{P}\Big(X(t)=0,\,Y(t)=ct\Big)+\int_{-ct}^{ct}e^{i\alpha\eta}f(\eta,\,t)\mathop{d\eta}\nonumber\\
&\qquad=\frac{e^{-\lambda t}}{2}\cos\left(\alpha ct\right)+\int_{-ct}^{ct}e^{i\alpha\eta}f(\eta,\,t)\mathop{d\eta}.\label{verifysec1}\end{align}
By using the expression (\ref{boundarydistrthm}) for the density $f$ and the integral formula \begin{equation}\int_{-ct}^{ct}e^{i\alpha \eta}\,I_0\left(\frac{\lambda}{c}\sqrt{c^2t^2-\eta^2}\right)\mathop{d\eta}=\frac{c}{\sqrt{\lambda^2-\alpha^2c^2}}\left[e^{t\sqrt{\lambda^2-\alpha^2c^2}}-e^{-t\sqrt{\lambda^2-\alpha^2c^2}}\right]\label{besselFourierT}\end{equation} it is a matter of straightforward calculation to prove that formula (\ref{verifysec1}) coincides with the characteristic function (\ref{boundarydistrchf}).

\section{Distribution on the diagonals}
\noindent If reflection is admitted, a degenerate component of the distribution of $\big(X(t),Y(t)\big)$ emerges on the diagonals of $S_{ct}$. In order to verify this, note that the process must start moving along a diagonal at time $t=0$. If the process only performs reflections, it is clear that it continues to move along the same diagonal as it only changes the travel orientation. Therefore, recalling the notation $$Q_{ct}=\left\{(x,y)\in S_{ct}:\;x=0\vee y=0\right\},$$ we have that
\begin{align}\mathbb{P}\Big(\big(X(t),Y&(t)\big)\in Q_{ct}\Big)=\sum_{k=0}^{\infty}\mathbb{P}\Big(\left(X(t),Y(t)\right)\in Q_{ct}\Big\lvert N(t)=k\Big)\,\mathbb{P}\left(N(t)=k\right)\nonumber\\
=&e^{-\lambda t}\sum_{k=0}^{\infty}\frac{(1-p-q)^k(\lambda t)^k}{k!}=e^{-\lambda (p+q)t}.\label{trenitaliatre}\end{align}
\noindent Observe that, if $p+q=1$, the expression (\ref{trenitaliatre}) corresponds to the probability of no changes of direction occurring. In other words, if reflection is not admitted the particle lies on the diagonals if and only if it never changes direction, in which case it must lie in one of the vertices of $S_{ct}$. We also emphasize that formula (\ref{trenitaliatre}) gives the probability of the particle belonging to whole diagonals, including the extremal points. In order to only consider the interior of the diagonals, the probability mass on the vertices of $S_{ct}$ must be subtracted. For instance, is we consider the horizontal diagonal, we have that \begin{align}\mathbb{P}\Big(Y(t)=0,\;-ct<X(t)<ct\Big)=\frac{1}{2}\left[\mathbb{P}\Big(\big(X(t),Y(t)\big)\in Q_{ct}\Big)-\mathbb{P}\big(N(t)=0\big)\right]=\frac{e^{-\lambda (p+q)t}-e^{-\lambda t}}{2}.\label{trenitaliaquattro}\end{align}

We now want to determine the exact distribution of the process on the diagonals $Q_{ct}$. Without loss of generality, we only consider the horizontal diagonal. Therefore we study the probability density function
\begin{equation}g(x,t)=\mathbb{P}\Big( X(t)\in\mathop{dx},\;Y(t)=0\Big)/\mathop{dx},\qquad \lvert x\lvert<ct.\label{diagonaldistr}\end{equation}
\noindent By defining, for $j=0,2$, the densities
\begin{equation*}g_j(x,t)=\mathbb{P}\Big( X(t)\in\mathop{dx},\;Y(t)=0,\;D(t)=d_j\Big)/\mathop{dx},\qquad \lvert x\lvert<ct\end{equation*}
it is clear that \begin{equation}g(x,t)=g_0(x,t)+g_2(x,t).\label{diagonalgsum}\end{equation}
\noindent Moreover, we can write
\begin{equation}\begin{cases}g_0(x,t+\mathop{dt})=g_0(x-c\mathop{dt},t)(1-\lambda\mathop{dt})+g_2(x,t)\lambda(1-p-q)\mathop{dt}+o(\mathop{dt})\\g_2(x,t+\mathop{dt})=g_2(x+c\mathop{dt},t)(1-\lambda\mathop{dt})+g_0(x,t)\lambda(1-p-q)\mathop{dt}+o(\mathop{dt})\end{cases}\label{diagonaleqs}\end{equation}
\noindent which implies that the system of partial differential equations \begin{equation}\begin{dcases}\frac{\partial g_0}{\partial t}=-c\frac{\partial g_0}{\partial x}+\lambda(1-p-q)g_2-\lambda g_0\\\frac{\partial g_2}{\partial t}=c\frac{\partial g_2}{\partial x}+\lambda(1-p-q)g_0-\lambda g_2\end{dcases}\label{diagonaldiffeqs}\end{equation} is satisfied with initial conditions $g_0(x,0)=g_2(x,0)=\frac{1}{4}\delta(x)$. The system (\ref{diagonaldiffeqs})  implies, in view of equation (\ref{diagonalgsum}), that
\begin{equation}\label{diagonalpde}\left(\frac{\partial^2}{\partial t^2}+2\lambda \frac{\partial}{\partial t}-c^2\frac{\partial^2}{\partial x^2}+\lambda^2(p+q)\left(2-p-q\right)\right)g=0\end{equation}
\noindent Therefore, we obtain an explicit expression for the density $g(x,t)$ in the following theorem.
\begin{thm}\label{thm:diagdistr}If $p+q<1$, the probability density function (\ref{diagonaldistr}) is
\begin{align*}\label{boundarydistrthm}g(x,t)=\frac{e^{-\lambda t}}{4c}\Bigg[\lambda(1-p-q)&\; I_0\left(\frac{\lambda}{c}(1-p-q)\sqrt{c^2t^2-x^2}\right)\\
+&\frac{\partial}{\partial t}I_0\left(\frac{\lambda}{c}(1-p-q)\sqrt{c^2t^2-x^2}\right)\Bigg],\qquad \vert x\lvert<ct.\end{align*}
\end{thm}
\begin{proof}We start by defining the auxiliary function $$\widetilde{g}(x,t)=e^{\lambda t}g(x,t).$$ Equation (\ref{diagonalpde}) implies that $\widetilde{g}(x,t)$ satisifes the partial differential equation
satisfies the partial differential equation
\begin{equation}\label{diagonalpdecheck}\frac{\partial^2\widetilde{g}}{\partial t^2}-c^2\frac{\partial^2\widetilde{g}}{\partial x^2}=\lambda^2(1-p-q)^2\widetilde{g}.\end{equation}
The change of variables \begin{equation}z=\sqrt{c^2t^2-x^2}\label{zvarg}\end{equation} transforms equation (\ref{diagonalpdecheck}) into the Bessel equation
\begin{equation*}\label{besselg}\frac{d^2\widetilde{g}}{z^2}+\frac{1}{z}\frac{d\widetilde{g}}{dz}-\frac{\lambda^2}{c^2}(1-p-q)^2\widetilde{g}=0\end{equation*}
whose general solution reads \begin{equation}\label{basselgsol}\widetilde{g}(z)=A\;I_0\left(\frac{\lambda}{c}(1-p-q)z\right)+B\;K_0\left(\frac{\lambda}{c}(1-p-q)z\right).\end{equation}
Similarly to theorem \ref{thm:f}, we disregard the term involving the modified Bessel function of the second kind $K_0(\cdot)$ and, since equation (\ref{diagonalpde}) is homogeneous with respect to the time variable, we express its solution in the form 
\begin{equation}\label{diagonalgensol}g(x,t)=e^{-\lambda t}\left[A\; I_0\left(\frac{\lambda}{c}(1-p-q)\sqrt{c^2t^2-x^2}\right)+B\;\frac{\partial}{\partial t}I_0\left(\frac{\lambda}{c}(1-p-q)\sqrt{c^2t^2-x^2}\right)\right]\end{equation}
In order to determine the coefficients $A$ and $B$ in formula (\ref{diagonalgensol}), we use formulas (\ref{IntI}) and (\ref{IntdI}) and we obtain
\begin{equation}\label{diagonalgensolint}\int_{-ct}^{ct}g(x,t)\mathop{dx}=e^{-\lambda t}\left\{\left(Bc+\frac{Ac}{\lambda(1-p-q)}\right)e^{\lambda t(1-p-q)}+\left(Bc-\frac{Ac}{\lambda(1-p-q)}\right)e^{-\lambda t(1-p-q)}-2Bc\right\}.\end{equation}
Comparing formulas (\ref{trenitaliaquattro}) and (\ref{diagonalgensolint}) yields
$$A=\frac{\lambda(1-p-q)}{4c},\qquad B=\frac{1}{4c}.$$
which completes the proof.
\end{proof}
\smallskip

We now derive the characteristic function of the process on on the horizontal diagonal of the support.

\begin{thm}\label{thm:diagchf}The characteristic function of $\big(X(t),Y(t)\big)$ on the horizontal diagonal of the square $S_{ct}$ is
\begin{align}\mathbb{E}\left[e^{i\alpha X(t)}\;\mathds{1}_{\{Y(t)=0\}}\right]=\frac{e^{-\lambda t}}{4}&\left[\left(1+\frac{\lambda(1-p-q)}{\sqrt{\lambda^2(1-p-q)^2-\alpha^2c^2}}\right)e^{t\sqrt{\lambda^2(1-p-q)^2-\alpha^2c^2}}\right.\nonumber\\+&\left.\left(1-\frac{\lambda(1-p-q)}{\sqrt{\lambda^2(1-p-q)^2-\alpha^2c^2}}\right)e^{-t\sqrt{\lambda^2(1-p-q)^2-\alpha^2c^2}}\right].\label{diagonaldistrchf}\end{align}
\end{thm}
\begin{proof}
In view of equation (\ref{diagonalpde}), the characteristic function $$\widehat{g}(\alpha,t)=\mathbb{E}\left[e^{i\alpha X(t)}\;\mathds{1}_{\{Y(t)=0\}}\right]$$
satisfies the ordinary differential equation
\begin{equation}\label{diagonalchfode}
\begin{dcases}
\frac{d^2\widehat{g}}{dt^2}+2\lambda \frac{d\widehat{g}}{dt}+\left[\lambda^2(p+q)\left(2-p-q\right)+\alpha^2c^2\right]\widehat{g}=0\\
\widehat{g}(\alpha,0)=\frac{1}{2}\\
\left.\frac{d}{d t}\widehat{g}(\alpha,t)\right\lvert_{t=0}=-\frac{\lambda(p+q)}{2}.
\end{dcases}\end{equation}
The general solution to equation (\ref{boundarychfode}) is
$$\widehat{g}(\alpha,t)=k_0(\alpha)\;e^{-\lambda t+t\sqrt{\lambda^2(1-p-q)^2-\alpha^2c^2}}+k_1(\alpha)\;e^{-\lambda t-t\sqrt{\lambda^2(1-p-q)^2-\alpha^2c^2}}.$$
\noindent The coefficients $k_0$ and $k_1$ can be obtained by using the initial conditions.
\end{proof}
We emphasize that the consistency of the results presented in this section can be proved by using a similar approach to that of section \ref{sec:boundary}. In particular, setting $\alpha=0$ in formula (\ref{diagonaldistrchf}) yields the probability (\ref{trenitaliatre}) while the characteristic function (\ref{diagonaldistrchf}) can be obtained direcly by using theorem \ref{thm:diagdistr} and the integral (\ref{besselFourierT}).

\section{Distribution in the interior of the domain}
\noindent In this section, we study the distribution of the process $\big(X(t),Y(t)\big)$ in the interior of its domain, that is the interior of the set
$$S_{ct}=\left\{(x,y)\in\mathbb{R}^2:\;\lvert x\lvert+\lvert y\lvert\le ct\right\}.$$
\noindent Therefore, we are interested in studying the function $u(x,y,t)$ defined by the relationship
\begin{equation}u(x,y,t)\mathop{dx}\mathop{dy}=\mathbb{P}\Big(X(t)\in\mathop{dx},\;Y(t)\in\mathop{dy}\Big),\qquad \lvert x\lvert+\lvert y\lvert< ct.\label{udef}\end{equation}
\noindent We emphasize that when the process admits reflection, that is when $p+q<1$, the distribution of $\big(X(t),Y(t)\big)$ has a degenerate component on the diagonals of the square $S_{ct}$. Therefore, if reflection is possible we study the function $u(x,y,t)$ under the implicit assumption that $x\neq0$ and $y\neq0$ since the definition (\ref{udef}) does not make sense otherwise.
\noindent In order to study the distribution (\ref{udef}) we consider, for $j=0,1,2,3$, the probability density functions $u_j(x,y,t)$ defined as
$$u_j(x,y,t)\mathop{dx}\mathop{dy}=\mathbb{P}\Big(X(t)\in\mathop{dx},\;Y(t)\in\mathop{dy},\;D(t)=d_j\Big).$$
\noindent By means of standard techniques it can be shown that the following linear system of equations is satisfied by $u_j(x,y,t),\;j=0,1,2,3$
\begin{equation}\label{interiordiffeqs}
\begin{dcases}
\dfrac{\partial u_0}{\partial t}=-c\dfrac{\partial u_0}{\partial x}+\lambda q u_1+\lambda (1-p-q) u_2+\lambda p u_3-\lambda u_0\\
\dfrac{\partial u_1}{\partial t}=-c\dfrac{\partial u_1}{\partial y}+\lambda p u_0+\lambda q u_2+\lambda (1-p-q) u_3-\lambda u_1\\
\dfrac{\partial u_2}{\partial t}=c\dfrac{\partial u_2}{\partial x}+\lambda (1-p-q) u_0+\lambda p u_1+\lambda q u_3-\lambda u_2\\
\dfrac{\partial u_3}{\partial t}=c\dfrac{\partial u_3}{\partial y}+\lambda q u_0+\lambda (1-p-q) u_1+\lambda p u_2-\lambda u_3
\end{dcases}
\end{equation}
\noindent with initial conditions $u_j(x,y,0)=\frac{1}{4}\,\delta(x)\,\delta(y)$ for $j=0,1,2,3$. Moreover, since $u(x,y,t)=\sum_{j=0}^3u_j(x,y,t)$, the system of equations (\ref{interiordiffeqs}) implies that $u(x,y,t)$ satisfies a partial differential equation which can be represented in the form
\begin{align}\Bigg\{\Bigg(\left(\frac{\partial}{\partial t}+\lambda\right)&^2-c^2\frac{\partial^2}{\partial x^2}\Bigg)\left(\left(\frac{\partial}{\partial t}+\lambda\right)^2-c^2\frac{\partial^2}{\partial y^2}\right)-4\lambda^2pq\left(\frac{\partial}{\partial t}+\lambda\right)^2\nonumber\\
-&\lambda^2(1-p-q)^2\left(2\left(\frac{\partial}{\partial t}+\lambda\right)^2-c^2\Delta\right)-4\lambda^3(p^2+q^2)(1-p-q)\left(\frac{\partial}{\partial t}+\lambda\right)\nonumber\\
\;&\qquad\quad+\lambda^4\left[(1-p-q)^2-2pq\right]^2-\lambda^4(p^2+q^2)^2\Bigg\}u=0.\label{upde}\end{align}
\begin{thm}\label{hydrolimthmbm}For $\lambda,c\to+\infty$ with $\frac{c^2}{\lambda}\to1$, the partial differential equation (\ref{upde}) becomes \begin{equation}\label{updelimit}\frac{\partial u}{\partial t}=\frac{1}{4}\;\frac{(1-p)+(1-q)}{(1-p)^2+(1-q)^2}\;\Delta u\end{equation}
\noindent where $\Delta=\frac{\partial^2}{\partial x^2}+\frac{\partial^2}{\partial y^2}$ is the Laplacian.\end{thm}
\begin{proof}
\noindent The theorem is proved by dividing equation (\ref{upde}) by $\lambda^3$, simplifying the expression for eliminating all the terms which grow as a multiple of $\lambda$ and taking the limit.
\end{proof}
The interpretation of theorem \ref{hydrolimthmbm} is that the process $\big(X(t),Y(t)\big)$ becomes, in the hydrodynamic limit, a planar Brownian motion with independent components and variance depending on $p$ and $q$. By taking into account the natural constraints to which $p$ and $q$ are subject, a constrained maximization problem can be solved in order to prove that the diffusion coefficient of equation (\ref{updelimit}) is maximized for $p=q=\frac{1}{2}$. Intuitively, this is due to the fact that both reflection and the vorticity effect induced by the condition $p\neq q$ tend to prevent the particle from moving too far from the starting point. The case $p=q=\frac{1}{2}$ is the only case in which both reflection and the vorticity effect are absent, which implies that the diffusion coefficient is maximized.\\
We now examine the special case in which the process $\big(X(t),Y(t)\big)$ does not admit reflection. By setting $q=1-p$, it can be easily verified that the partial differential equation (\ref{upde}) reduces to
\begin{equation}\label{norefpde}\left(\left(\frac{\partial}{\partial t}+\lambda\right)^4-\left(c^2\Delta+4\lambda^2p(1-p)\right)\left(\frac{\partial}{\partial t}+\lambda\right)^2+c^2\frac{\partial^4}{\partial x^2\partial y^2}\right)u=\lambda^4(2p-1)^2u.\end{equation}
\noindent Even in this simpler case, finding the distribution of $\big(X(t),Y(t)\big)$ in closed form is a difficult task. However, we are able to obtain an explicit form of the characteristic function as shown in the following theorem.
\begin{thm}If $p+q=1$, the characteristic function of $\big(X(t),Y(t)\big)$ in the interior of $\partial S_{ct}$ is
\begin{align}\mathbb{E}\Big[&e^{i\left(\alpha X(t)+\beta Y(t)\right)}\Big]\nonumber\\
&\;\;\;\;=e^{-\lambda t}\Bigg\{k_0(\alpha,\beta)\,\cosh\Big(A(\alpha,\beta)\mathop{t}+B(\alpha,\beta)\mathop{t}\Big)+k_1(\alpha,\beta)\,\sinh\Big(A(\alpha,\beta)\mathop{t}+B(\alpha,\beta)\mathop{t}\Big)\nonumber\\
&\qquad\;\;\;\;-k_2(\alpha,\beta)\,\cosh\Big(A(\alpha,\beta)\mathop{t}-B(\alpha,\beta)\mathop{t}\Big)-k_3(\alpha,\beta)\,\sinh\Big(A(\alpha,\beta)\mathop{t}-B(\alpha,\beta)\mathop{t}\Big)\Bigg\}\label{interiordistrchf}\end{align}
\noindent where 
\begin{equation}A(\alpha,\beta)=\frac{1}{2}\sqrt{4\lambda^2p(1-p)-c^2(\alpha^2+\beta^2)+2\sqrt{c^4\alpha^2\beta^2-\lambda^4(2p-1)^2}}\label{A}\end{equation}
\begin{equation}B(\alpha,\beta)=\frac{1}{2}\sqrt{4\lambda^2p(1-p)-c^2(\alpha^2+\beta^2)-2\sqrt{c^4\alpha^2\beta^2-\lambda^4(2p-1)^2}}\label{B}\end{equation}
\noindent and 
\begin{align*}
k_0(\alpha,\beta)=&\frac{2\lambda^2-c^2(\alpha^2+\beta^2)-2\left(A(\alpha,\beta)-B(\alpha,\beta)\right)^2}{8A(\alpha,\beta)B(\alpha,\beta)}\nonumber\\
k_1(\alpha,\beta)=&\frac{\lambda^3-\lambda c^2(\alpha^2+\beta^2)-\lambda\left(A(\alpha,\beta)-B(\alpha,\beta)\right)^2}{4A(\alpha,\beta)B(\alpha,\beta)\left(A(\alpha,\beta)+B(\alpha,\beta)\right)}\nonumber\\
k_2(\alpha,\beta)=&\frac{2\lambda^2-c^2(\alpha^2+\beta^2)-2\left(A(\alpha,\beta)+B(\alpha,\beta)\right)^2}{8A(\alpha,\beta)B(\alpha,\beta)}\nonumber\\
k_3(\alpha,\beta)=&\frac{\lambda^3-\lambda c^2(\alpha^2+\beta^2)-\lambda\left(A(\alpha,\beta)+B(\alpha,\beta)\right)^2}{4A(\alpha,\beta)B(\alpha,\beta)\left(A(\alpha,\beta)-B(\alpha,\beta)\right)}.\nonumber
\end{align*}
\end{thm}
\begin{proof}We adopt the notation \begin{equation}\widehat{u}(\alpha,\beta,t)=\mathbb{E}\Big[e^{i\left(\alpha X(t)+\beta Y(t)\right)}\Big]\label{interiorchfnotation}.\end{equation}
In view of equation (\ref{norefpde}), the characteristic function $\widehat{u}(\alpha,\beta,t)$ satisfies the ordinary differential equation
\begin{equation}\label{norefode}\left(\left(\frac{d}{d t}+\lambda\right)^4+\left(c^2(\alpha^2+\beta^2)-4\lambda^2p(1-p)\right)\left(\frac{d}{d t}+\lambda\right)^2+c^2\alpha^2\beta^2\right)\widehat{u}=\lambda^4(2p-1)^2\widehat{u}.\end{equation}
\noindent We now establish the initial conditions for equation (\ref{norefode}). For this purpose we define the functions $$\widehat{u}_j(\alpha,\beta,t)=\mathbb{E}\Big[e^{i\left(\alpha X(t)+\beta Y(t)\right)}\;\mathds{1}_{\{D(t)=d_j\}}\Big],\qquad j=0,1,2,3$$
\noindent and, in view of the system of equations (\ref{interiordiffeqs}), we observe that for all $n\in\mathbb{N}$
\begin{equation}\label{eigen}
\frac{d^n}{d t^n}\begin{pmatrix}
 \widehat{u}_0\\
 \widehat{u}_1\\
 \widehat{u}_2\\
 \widehat{u}_3
\end{pmatrix}
=\text{L}^n\cdot\begin{pmatrix}
 \widehat{u}_0\\
 \widehat{u}_1\\
 \widehat{u}_2\\
 \widehat{u}_3
\end{pmatrix}
\end{equation}
\noindent where $$\text L=
\begin{pmatrix}
i\alpha c-\lambda&\lambda(1-p)&0&\lambda p\\
\lambda p&i\beta c-\lambda& \lambda (1-p)& 0 \\
0 & \lambda p&-i\alpha c-\lambda &\lambda (1-p)&  \\
\lambda (1-p)& 0&\lambda p &-i\beta c-\lambda
\end{pmatrix}.$$
\noindent Equation (\ref{eigen}) can be proved for $n=1$ by taking the Fourier transforms of the equations in the system (\ref{interiordiffeqs}) and can be extended to general values of $n$ by induction. Moreover, formula (\ref{eigen}) implies that
\begin{equation}\label{eigen0}
\frac{d^n}{d t^n}\left.\begin{pmatrix}
 \widehat{u}_0\\
 \widehat{u}_1\\
 \widehat{u}_2\\
 \widehat{u}_3
\end{pmatrix}\right\rvert_{t=0}
=\frac{1}{4}\;\text{L}^n\cdot\begin{pmatrix}
1\\
1\\
1\\
1
\end{pmatrix}
\end{equation}
\noindent where we have used the fact that, for $j=0,1,2,3$, the condition $u_j(x,y,0)=\frac{1}{4}\delta(x)\delta(y)$ implies that $\widehat{u}_j(\alpha,\beta,0)=\frac{1}{4}$. By using formula (\ref{eigen0}) and by taking into account that $\widehat{u}(\alpha,\beta,t)=\sum_{j=0}^3\widehat{u}_j(\alpha,\beta,t)$, we obtain the following initial conditions for equation (\ref{norefode}):
\begin{equation}\begin{alignedat}{2}
\widehat{u}&(\alpha,\beta,0)=1,\qquad\qquad&&\;\frac{d}{dt}\widehat{u}(\alpha,\beta,t)\Big\rvert_{t=0}=0,\\
\frac{d^2}{dt^2}\widehat{u}&(\alpha,\beta,t)\Big\rvert_{t=0}=-\frac{c^2(\alpha^2+\beta^2)}{2},\qquad\qquad&& \frac{d^3}{dt^3}\widehat{u}(\alpha,\beta,t)\Big\rvert_{t=0}=\frac{\lambda c^2(\alpha^2+\beta^2)}{2}.\end{alignedat}\label{interiorincond}\end{equation}

\noindent We now define the function \begin{equation}\label{interiortilde}\widetilde{u}(\alpha,\beta,t)=e^{\lambda t}\;\widehat{u}(\alpha,\beta,t)\end{equation}
\noindent and we observe that, in view of formulas (\ref{norefode}) and (\ref{interiorincond}), the following ordinary differential equation is satisfied by $\widetilde{u}(\alpha,\beta,t)$:
\begin{equation}\label{eulerode}
\begin{dcases}
\dfrac{d^4\widetilde{u}}{d t^4}+\left(c^2(\alpha^2+\beta^2)-4\lambda^2p(1-p)\right)\dfrac{d^2\widetilde{u}}{d t^2}+\left(c^2\alpha^2\beta^2-\lambda^4(2p-1)^2\right)\widetilde{u}=0\\
\widetilde{u}(\alpha,\beta,0)=1\\
\dfrac{d}{dt}\widetilde{u}(\alpha,\beta,t)\Big\rvert_{t=0}=\lambda\\
\dfrac{d^2}{dt^2}\widetilde{u}(\alpha,\beta,t)\Big\rvert_{t=0}=\lambda^2-\frac{c^2(\alpha^2+\beta^2)}{2}\\
\dfrac{d^3}{dt^3}\widetilde{u}(\alpha,\beta,t)\Big\rvert_{t=0}=\lambda^3-\lambda c^2(\alpha^2+\beta^2).
\end{dcases}
\end{equation}
\noindent Equation (\ref{eulerode}) is an Euler-type ordinary differential equation which can be solved by finding the roots of the associated algebraic equation
$$r^4+\left(c^2(\alpha^2+\beta^2)-4\lambda^2p(1-p)\right)r^2+\left(c^2\alpha^2\beta^2-\lambda^4(2p-1)^2\right)=0.$$
\noindent Since the above algebraic equation is biquadratic, its roots can be easily found and it can be shown that the general solution to equation (\ref{eulerode}) can be expressed in the form
\begin{align}\widetilde{u}(\alpha,\beta,t)=\varphi_0(&\alpha,\beta)\,e^{A(\alpha,\beta)t+B(\alpha,\beta)t}+\varphi_1(\alpha,\beta)\,e^{A(\alpha,\beta)t-B(\alpha,\beta)t}\nonumber\\
\;\qquad+&\varphi_2(\alpha,\beta)\,e^{-A(\alpha,\beta)t-B(\alpha,\beta)t}+\varphi_3(\alpha,\beta)\,e^{-A(\alpha,\beta)t+B(\alpha,\beta)t}\label{generalsolalg}\end{align}
\noindent where $A(\alpha,\beta)$ and $B(\alpha,\beta)$ are defined in formulas (\ref{A}) and (\ref{B}). By taking the derivatives of the general solution (\ref{generalsolalg}) and using the initial conditions of equation (\ref{eulerode}), it can be proved that the coefficients $\varphi_j(\alpha,\beta),\;j=0,1,2,3$, satisfy the linear system of equations

\begin{equation}
\begin{pmatrix}
1&1&1&1\\
A+B&A-B& -A-B&-A+B\\
(A+B)^2&(A-B)^2& (A+B)^2&(A-B)^2\\
(A+B)^3&(A-B)^3& -(A+B)^3&-(A-B)^3\\
\end{pmatrix}
\begin{pmatrix}
\varphi_0\\\varphi_1\\\varphi_2\\\varphi_3
\end{pmatrix}=\begin{pmatrix}
1\\\lambda\\\lambda^2-\frac{c^2(\alpha^2+\beta^2)}{2}\\[0.1cm]\lambda^3-\lambda c^2(\alpha^2+\beta^2)
\end{pmatrix}
\label{ABtoinvert}\end{equation}
\noindent where we omitted, for simplicity, the dependence of $A, B, \varphi_0, \varphi_1, \varphi_2$ and $\varphi_3$ on $\alpha$ and $\beta$. The linear system (\ref{ABtoinvert}) can be solved by inverting the coefficient matrix. Thus, we have that
\begin{equation}
\begin{pmatrix}
\varphi_0\\\varphi_1\\\varphi_2\\\varphi_3
\end{pmatrix}=
\begin{pmatrix}
-\frac{(A-B)^2}{8AB}&-\frac{(A-B)^2}{8AB(A+B)}&\frac{1}{8AB}&\frac{1}{8AB(A+B)}\\[0.1cm]
\frac{(A+B)^2}{8AB}&\frac{(A+B)^2}{8AB(A-B)}&-\frac{1}{8AB}& -\frac{1}{8AB(A-B)}\\[0.1cm]
-\frac{(A-B)^2}{8AB}&\frac{(A-B)^2}{8AB(A+B)}&\frac{1}{8AB}&-\frac{1}{8AB(A+B)}\\[0.1cm]
\frac{(A+B)^2}{8AB}&-\frac{(A+B)^2}{8AB(A-B)}&-\frac{1}{8AB}& \frac{1}{8AB(A-B)}
\end{pmatrix}
\begin{pmatrix}
1\\\lambda\\\lambda^2-\frac{c^2(\alpha^2+\beta^2)}{2}\\[0.1cm]\lambda^3-\lambda c^2(\alpha^2+\beta^2)
\end{pmatrix}
\end{equation}
\noindent from which the coefficients $\varphi_j(\alpha,\beta),\;j=0,1,2,3$, are immediately obtained. In particular, it can be shown that
\begin{align}\varphi_0(\alpha,\beta)=\frac{k_0(\alpha,\beta)+k_1(\alpha,\beta)}{2},\qquad\varphi_1(\alpha,\beta)=-\frac{k_2(\alpha,\beta)+k_3(\alpha,\beta)}{2},\nonumber\\
\varphi_2(\alpha,\beta)=\frac{k_0(\alpha,\beta)-k_1(\alpha,\beta)}{2},\qquad\varphi_3(\alpha,\beta)=-\frac{k_2(\alpha,\beta)-k_3(\alpha,\beta)}{2},\nonumber
\end{align}
where the coefficients $k_j(\alpha,\beta),\;j=0,1,2,3$, are defined in the statement of the theorem. After substituting the expressions of $\varphi_j(\alpha,\beta),\;j=0,1,2,3$, into formula (\ref{generalsolalg}), the proof is completed by inverting the transformation (\ref{interiortilde}) and properly collecting the terms of the final expression.
\end{proof}

\noindent We observe that, although the characteristic function (\ref{interiordistrchf}) has a cumbersome representation, it is consistent with the existing literature since it coincides with the results obtained by Orsingher \cite{orsingher2000} for $p=q=\frac{1}{2}$. Moreover, we emphasize that the key step for obtaining the characteristic function was to solve the fourth-order algebraic equation (\ref{eulerode}). Such equation is easy to solve because it is biquadratic in the case $p+q=1$. In principle, the characteristic function of $\big(X(t),Y(t)\big)$ could also be obtained in the case in which reflection is admitted by solving a more general fourth-order equation. However, since the equation is not biquadratc for $p+q<1$, the resulting characteristic function would have an extremely complicated form.

\section{Time spent in vertical direction}
\noindent We conclude our analysis by studying the time spent by the process $\big(X(t),Y(t)\big)$ moving vertically, parallel to the $y$-axis. Formally, we define the process
\begin{equation*}\label{Tdef}T(t)=\int_0^t\mathds{1}_{\{D(\tau)\in\{d_1,d_3\}\}}\mathop{d\tau},\qquad t>0.\end{equation*}
For fixed $t>0$, the support of the random variable $T(t)$ coincides with the interval $[0,t]$. Moreover, it is clear that the distribution of $T(t)$ has a degenerate component on the extrema of its support. In particular, we have that $T(t)=0$ if the process $\big(X(t),Y(t)\big)$ starts moving in horizontal direction at time 0 and it never changes direction or it changes direction by only performing reflections until time $t$. Similarly, it holds that $T(t)=t$ if the process $\big(X(t),Y(t)\big)$ starts moving vertically at time 0 and it never changes direction or it only performs reflections until time $t$. Therefore, we have that
\begin{equation}\label{verticalboundary}\mathbb{P}\left(T(t)=0\right)=\mathbb{P}\left(T(t)=t\right)=\frac{1}{2}\;e^{-\lambda t}\sum_{k=0}^{\infty}\frac{(\lambda t)^k(1-p-q)^k}{k!}=\frac{1}{2}\;e^{-\lambda t(p+q)}.\end{equation}
\noindent In the interior of its support, the random variable $T(t)$ has a continuous distribution. We denote the corresponding probability density function by
\begin{equation}\label{hdef}h(s,t)=\mathbb{P}\Big(T(t)\in\mathop{ds}\Big)/ds,\qquad s\in(0,t).\end{equation}
\noindent In order to determine the density (\ref{hdef}), we start by defining the functions $h_j(s,t),\;j=0,1,$
\begin{equation*}h_j(s,t)=\mathbb{P}\Big(T(t)\in\mathop{ds},\;D(t)\in\{d_j,\;d_{j+2}\}\Big)/ds,\qquad s\in(0,t).\end{equation*}
\noindent It is clear that
\begin{equation*}
\begin{dcases}
h_0(s,t+\mathop{dt})=h_0(s,t)(1-\lambda(p+q)\mathop{dt})+h_1(s-\mathop{ds},t)\lambda(p+q)\mathop{dt}+o(dt)\\
h_1(s,t+\mathop{dt})=h_1(s-\mathop{ds},t)(1-\lambda(p+q)\mathop{dt})+h_0(s,t)\lambda(p+q)\mathop{dt}+o(dt)
\end{dcases}
\end{equation*}
\noindent which implies that the system of differential equations 
\begin{equation}\label{verticaldiffeqs}
\begin{dcases}
\dfrac{\partial h_0}{\partial t}=\lambda(p+q)\left(h_1-h_0\right)\\
\dfrac{\partial h_1}{\partial t}=-\dfrac{\partial h_1}{\partial s}+\lambda(p+q)\left(h_0-h_1\right)\\
\end{dcases}
\end{equation}
\noindent is satisfied with initial conditions $h_0(s,0)=h_1(s,0)=\frac{1}{2}\;\delta(s)$. In view of the system (\ref{verticaldiffeqs}) and taking into account that $$h(s,t)=h_0(s,t)+h_1(s,t)$$ the density function $h(s,t)$ satisfies the partial differential equation
\begin{equation}\label{hpde}\left(\frac{\partial^2}{\partial t^2}+\frac{\partial^2}{\partial s\;\partial t}+2\lambda(p+q)\frac{\partial}{\partial t}+\lambda(p+q)\frac{\partial}{\partial s}\right)h=0.\end{equation}
\noindent Thus, we are able to prove the following result.
\begin{thm}The probability density function (\ref{hdef}) reads
\begin{equation}\label{verticaldistrthm}h(s,t)=e^{-\lambda(p+q) t}\left[\lambda(p+q)\; I_0\left(2\lambda(p+q)\sqrt{s(t-s)}\right)+\frac{\partial}{\partial t}I_0\left(2\lambda(p+q)\sqrt{s(t-s)}\right)\right],\qquad s\in(0,t).\end{equation}\end{thm}
\begin{proof}In light of equation (\ref{hpde}), the function $$\widetilde{h}(s,t)=e^{\lambda(p+q)t}\;h(s,t)$$ satisfies the partial differential equation
\begin{equation}\label{hcheckpde}\left(\frac{\partial^2}{\partial t^2}+\frac{\partial^2}{\partial s\partial t}-\lambda^2(p+q)^2\right)\widetilde{h}=0.\end{equation} The change of variables $$z=\sqrt{s(t-s)}$$ transforms equation (\ref{hcheckpde}) into the Bessel equation
\begin{equation*}\label{besselh}\frac{d^2\widetilde{h}}{z^2}+\frac{1}{z}\frac{d\widetilde{h}}{dz}-4\lambda^2(p+q)^2\widetilde{h}=0\end{equation*}
whose general solution reads \begin{equation*}\label{besselhsol}\widetilde{h}(z)=A\;I_0\Big(2\lambda(p+q)z\Big)+B\;K_0\Big(2\lambda(p+q)z\Big).\end{equation*}
\noindent Therefore, usual arguments permit us to express the solution to equation (\ref{hpde}) in the form
\begin{equation}\label{verticaldistrthmgeneral}h(s,t)=e^{-\lambda(p+q) t}\left[A\; I_0\left(2\lambda(p+q)\sqrt{s(t-s)}\right)+B\;\frac{\partial}{\partial t}I_0\left(2\lambda(p+q)\sqrt{s(t-s)}\right)\right].\end{equation}
\noindent In order for formula (\ref{verticaldistrthmgeneral}) to be consistent with (\ref{verticalboundary}), we must have that $$\int_0^th(s,t)\mathop{ds}=1-e^{-\lambda(p+q)t}.$$ By using the relation \begin{equation}\label{s(t-s)besselintegral}\int_0^tI_0\Big(K\;\sqrt{s(t-s)}\Big)\mathop{ds}=\frac{1}{K}\left(e^{\frac{Kt}{2}}-e^{-\frac{Kt}{2}}\right),\qquad t>0\end{equation} we obtain $$A=\lambda(p+q),\qquad B=1$$ which completes the proof.
\end{proof}

\noindent The corresponding characteristic function is given in the following theorem.

\begin{thm}\label{trenitalia69}The characteristic function of $T(t)$ is \begin{align}\mathbb{E}\left[e^{i\alpha T(t)}\right]=&\frac{1}{2}\left(1+\frac{2\lambda(p+q)}{\sqrt{4\lambda^2(p+q)^2-\alpha^2}}\right)e^{i \frac{\alpha}{2}t-\lambda(p+q)t+\frac{1}{2}\sqrt{4\lambda^2(p+q)^2-\alpha^2}\;t}\nonumber\\
\;&+\frac{1}{2}\left(1-\frac{2\lambda(p+q)}{\sqrt{4\lambda^2(p+q)^2-\alpha^2}}\right)e^{i \frac{\alpha}{2}t-\lambda(p+q)t-\frac{1}{2}\sqrt{4\lambda^2(p+q)^2-\alpha^2}\;t}.\label{verticalchf}\end{align}\end{thm}
\begin{proof}The function $$\widehat{h}(\alpha,t)=\mathbb{E}\left[e^{i\alpha T(t)}\right]$$ satisfies, in view of equation (\ref{hpde}), the ordinary differential equation
\begin{equation}\label{verticalchfode}
\begin{dcases}
\frac{d^2\widehat{h}}{dt^2}+\left[2\lambda(p+q)-i\alpha\right] \frac{d\widehat{h}}{dt}-i\lambda(p+q)\alpha\widehat{h}=0\\
\widehat{h}(\alpha,0)=1\\
\frac{d}{d t}\left.\widehat{h}(\alpha,t)\right\lvert_{t=0}=\frac{i\alpha}{2}.
\end{dcases}\end{equation}
The general solution to equation (\ref{verticalchfode}) can be expressed in the form $$\widehat{h}(\alpha,t)=k_0(\alpha)\;e^{i \frac{\alpha}{2}t-\lambda(p+q)t+\frac{1}{2}\sqrt{4\lambda^2(p+q)^2-\alpha^2}\;t}+k_1(\alpha)\;e^{i \frac{\alpha}{2}t-\lambda(p+q)t-\frac{1}{2}\sqrt{4\lambda^2(p+q)^2-\alpha^2}\;t}.$$ The coefficients $k_0$ and $k_1$ can be determined by using the initial conditions.
\end{proof}
\noindent Observe that the characteristic function (\ref{verticalchf}) can be computed directly from the density function (\ref{verticaldistrthm}). By writing that $$\mathbb{E}\left[e^{i\alpha T(t)}\right]=\int_0^te^{i\alpha s}\,h(s,t)\mathop{ds}$$ the integral formula $$\int_0^te^{i\alpha s}\,I_0\left(\beta\sqrt{s(t-s)}\right)\mathop{ds}=\frac{e^{\frac{i\alpha t}{2}}}{\beta^2-\alpha^2}\left[e^{\frac{t}{2}\sqrt{\beta^2-\alpha^2}}-e^{-\frac{t}{2}\sqrt{\beta^2-\alpha^2}}\right]$$ yields the result.
\smallskip

We conclcude our analysis by discussing some aspects of the joint distribution of the time spent moving vertically $T(t)$ and the component $Y(t)$ of the planar random motion studied so far. Clearly, the support of the bivariate process $\big(T(t),\;Y(t)\big)$ coincides with the triangle \begin{equation}R_{ct}=\left\{(s,y)\in\mathbb{R}^2:\;0\le s\le t,\;-ct\le y\le ct\right\}.\label{jointYTsupp}\end{equation} We consider the probability density function $\bar{u}(s,y,t)$ defined by the relationship\begin{equation}\label{jointYT}\bar{u}(s,y,t)\mathop{ds}\mathop{dy}=\mathbb{P}\Big(T(t)\in \mathop{ds},\;Y(t)\in\mathop{dy}\Big),\qquad (s,y)\in R_{ct}.\end{equation}
 At the initial time $t=0$, the process $\big(T(t),\;Y(t)\big)$ starts at the vertex $(0,0)$ of the triangle $R_{ct}$ and it moves along two directions which are parallel to the oblique sides of $R_{ct}$. In particular, the normalized vectors which describe the directions along which $\big(T(t),\;Y(t)\big)$ can move are given by $$\theta_+=\left(\frac{1}{\sqrt{1+c^2}},\;\frac{c}{\sqrt{1+c^2}}\right),\qquad \theta_-=\left(\frac{1}{\sqrt{1+c^2}},\;-\frac{c}{\sqrt{1+c^2}}\right).$$
Clearly, in both cases the particle with position $\big(T(t),\;Y(t)\big)$ moves rightwards since the process $T(t)$ is non-decreasing. We denote by $\Theta(t)$ the direction of $\big(T(t),\;Y(t)\big)$ at time $t>0$. In order to understand how the motion of $\big(X(t),\;Y(t)\big)$ on $S_{ct}$ and that of $\big(T(t),\;Y(t)\big)$ on $R_{ct}$ relate to each other, we first observe that $\Theta(t)=\theta_+$ if and only if $D(t)=d_1$ and $\Theta(t)=\theta_-$ if and only if $D(t)=d_3$. It is now clear that when $\big(X(t),\;Y(t)\big)$ moves horizontally, neither $T(t)$ increases nor $Y(t)$, so that the particle in $R_{ct}$ remains still. If the process $\big(T(t),\;Y(t)\big)$ is not moving at time $t$ we say that $\Theta(t)=0$. Therefore, we have that $\Theta(t)=0$ if and only if $D(t)\in\{d_0,d_2\}$.\\
\noindent We now study the rule which determines the changes of the direction $\Theta(t)$ for the particle moving in the triangle $R_{ct}$. For this purpose, we denote by $E_t$ the event of a change of direction occurring during the time interval $(t,t+\mathop{dt}]$. This enables us to write that
\begin{equation}\Theta(t+dt)\Big\lvert\left\{E_t,\;\Theta(t)=\theta_+\right\}\;=\begin{cases}0&\text{with prob.}\;p+q\\\theta_-&\text{with prob.}\;1-p-q.\end{cases}\label{theta+change}\end{equation} Formula (\ref{theta+change}) can be interpreted in the following manner. Assume that at time $t$ the particle $\big(T(t),\;Y(t)\big)$ is moving with direction $\Theta(t)=\theta_+$ and that the event $E_t$ occurs. Since $\Theta(t)=\theta_+$, we have that $D(t)=d_1$ and thus we can distinghuish between two cases. In the first case reflection occurs with probability $1-p-q$, which implies that $D(t+dt)=d_3$ and thus $\Theta(t)=\theta_-$. If reflection does not occur, which happens with probability $p+q$, we have that $D(t)\in\{d_0,d_2\}$ and therefore $\Theta(t+dt)=0$.
\noindent Similarly, we have that
\begin{equation*}\Theta(t+dt)\Big\lvert\left\{E_t,\;\Theta(t)=\theta_-\right\}\;=\begin{cases}0&\text{with prob.}\;p+q\\\theta_+&\text{with prob.}\;1-p-q.\end{cases}\end{equation*}
We now emphasize that difficulties emerge when conditioning the distribution of $\Theta(t+dt)$ on the event $\Theta(t)=0$. It can be easily verified that
\begin{equation}\label{T0+}\Theta(t+dt)\Big\lvert\left\{E_t,\;\Theta(t)=0,\;D(t)=d_0\right\}\;=\begin{cases}0&\text{with prob.}\;1-p-q\\\theta_+&\text{with prob.}\;p\\\theta_-&\text{with prob.}\;q.\end{cases}\end{equation} and analogously \begin{equation}\label{T0-}\Theta(t+dt)\Big\lvert\left\{E_t,\;\Theta(t)=0,\;D(t)=d_2\right\}\;=\begin{cases}0&\text{with prob.}\;1-p-q\\\theta_+&\text{with prob.}\;q\\\theta_-&\text{with prob.}\;p.\end{cases}\end{equation} The difference between formulas (\ref{T0+}) and (\ref{T0-}) shows that the distribution of $\Theta(t+dt)$, conditional on the fact that $\Theta(t)=0$ and that a change of direction occurs, cannot be determined exactly unless the additional information on $D(t)$ is given. The interesting fact about this phenomenon is that it only occurs if $p\neq q$. For $p=q$, the right-hand sides of formulas (\ref{T0+}) and (\ref{T0-}) are equal and we can write \begin{equation*}\Theta(t+dt)\Big\lvert\left\{E_t,\;\Theta(t)=0\right\}\;=\begin{cases}0&\text{with prob.}\;1-2p\\\theta_+&\text{with prob.}\;p\\\theta_-&\text{with prob.}\;p\end{cases}\end{equation*} which noticeably simplifies the joint distribution of the couple $\big(T(t),\;Y(t)\big)$. In order to understand how the study of the density function $\bar{u}(s,y,t)$ is simplified for $p=q$, we define, for $j=0,1,2,3$, the auxiliary density functions \begin{equation*}\bar{u}_j(s,y,t)\mathop{ds}\mathop{dy}=\mathbb{P}\left(T(t)\in \mathop{ds},\;Y(t)\in\mathop{dy},\;D(t)=d_j\right),\qquad (s,y)\in R_{ct}.\end{equation*} and we observe that the following linear system of partial differential equations is satisfied:
\begin{equation}\label{hbarinteriordiffeqs}
\begin{dcases}
\dfrac{\partial \bar{u}_0}{\partial t}=-\lambda \bar{u}_0+\lambda q \bar{u}_1+\lambda (1-p-q) \bar{u}_2+\lambda p \bar{u}_3\\
\dfrac{\partial \bar{u}_1}{\partial t}=-\dfrac{\partial \bar{u}_1}{\partial s}-c\dfrac{\partial \bar{u}_1}{\partial y}+\lambda p \bar{u}_0+\lambda q \bar{u}_2+\lambda (1-p-q) \bar{u}_3-\lambda \bar{u}_1\\
\dfrac{\partial \bar{u}_2}{\partial t}=-\lambda \bar{u}_2+\lambda (1-p-q) \bar{u}_0+\lambda p \bar{u}_1+\lambda q \bar{u}_3\\
\dfrac{\partial \bar{u}_3}{\partial t}=-\dfrac{\partial \bar{u}_3}{\partial s}+c\dfrac{\partial \bar{u}_3}{\partial y}+\lambda q \bar{u}_0+\lambda (1-p-q) \bar{u}_1+\lambda p \bar{u}_2-\lambda \bar{u}_3.
\end{dcases}
\end{equation}
\noindent The system (\ref{hbarinteriordiffeqs}) implies that $\bar{u}(s,y,t)$ satisfies the cumbersome fourth-order partial differential equation
\begin{align}&\Bigg\{\left[\left(\frac{\partial}{\partial t}+\lambda\right)^2-\lambda^2(1-p-q)^2\right]\left[\left(\frac{\partial}{\partial t}+\lambda+\frac{\partial}{\partial s}\right)^2-c^2\frac{\partial^2}{\partial y^2}\right]-4\lambda^2 p q \left(\frac{\partial}{\partial t}+\lambda\right)\left(\frac{\partial}{\partial t}+\lambda+\frac{\partial}{\partial s}\right)\nonumber\\
&\qquad-4\lambda^3(1-p-q)(p^2+q^2)\left(\frac{\partial}{\partial t}+\lambda+\frac{1}{2}\frac{\partial}{\partial s}\right)-\lambda^2(1-p-q)^2\left(\frac{\partial}{\partial t}+\lambda\right)^2-\lambda^4(p^2-q^2)^2\nonumber\\
&\qquad-4\lambda^4pq(1-p-q)^2+\lambda^4(1-p-q)^4\Bigg\}\bar{u}=0.\label{hbarpdeint}\end{align}
\noindent We now show that the system (\ref{hbarinteriordiffeqs}) and equation (\ref{hbarpdeint}) simplify in the case $p=q$. We define the function
\begin{equation*}\bar{u}_{02}(s,y,t)\mathop{ds}\mathop{dy}=\mathbb{P}\left(T(t)\in \mathop{ds},\;Y(t)\in\mathop{dy},\;D(t)\in\{d_0,\;d_2\}\right),\qquad s\in[0,t],\;y\in[-ct,ct].\end{equation*} Of course, we have that $$\bar{u}_{02}(s,y,t)=\bar{u}_{0}(s,y,t)+\bar{u}_{2}(s,y,t).$$ Therefore, the first and the third equations of the system (\ref{hbarinteriordiffeqs}) imply that, if $p=q$, the ordinary differential equation \begin{equation}\label{02ode}\dfrac{\partial \bar{u}_{02}}{\partial t}=-2\lambda p\,\bar{u}_{02}+2\lambda p (\bar{u}_1+\bar{u}_3)\end{equation} holds which permits to simplify the system (\ref{hbarinteriordiffeqs}) into
\begin{equation}\label{hbarinteriordiffeqs02}
\begin{dcases}
\dfrac{\partial \bar{u}_{02}}{\partial t}=-2\lambda p\, \bar{u}_{02}+2\lambda p (\bar{u}_1+\bar{u}_3)\\
\dfrac{\partial \bar{u}_1}{\partial t}=-\dfrac{\partial \bar{u}_1}{\partial s}-c\dfrac{\partial \bar{u}_1}{\partial y}+\lambda p\bar{u}_{02}+\lambda (1-2p) \bar{u}_3-\lambda \bar{u}_1\\
\dfrac{\partial \bar{u}_3}{\partial t}=-\dfrac{\partial \bar{u}_3}{\partial s}+c\dfrac{\partial \bar{u}_3}{\partial y}+\lambda p \bar{u}_{02}+\lambda (1-2p) \bar{u}_1-\lambda \bar{u}_3.
\end{dcases}
\end{equation}
\noindent Observe that equation (\ref{02ode}) is not satisfied if $p\neq q$. Since $$\bar{u}(s,y,t)=\bar{u}_{02}(s,y,t)+\bar{u}_{1}(s,y,t)+\bar{u}_{3}(s,y,t)$$ the linear system (\ref{hbarinteriordiffeqs02}) implies that, if $p=q$, equation (\ref{hbarpdeint}) reduces to the third-order partial differential equation
\begin{align}&\Bigg\{\left(\frac{\partial}{\partial t}+2\lambda p\right)\left[\left(\frac{\partial }{\partial t}+\lambda +\frac{\partial }{\partial s}\right)-c^2\frac{\partial^2}{\partial y^2}\right]-4\lambda^2 p^2\left(\frac{\partial }{\partial t}+\lambda +\frac{\partial }{\partial s}\right)\nonumber\\
&\qquad\qquad-\lambda^2(1-2p)^2\left(\frac{\partial }{\partial t}+2\lambda p\right)-4\lambda^3p^2(1-2p)\Bigg\}\bar{u}=0.\label{hbarpdeintorder3}\end{align}
\noindent Unfortunately, finding an explicit solution to equation (\ref{hbarpdeint}) seems a difficult task even in the special case (\ref{hbarpdeintorder3}). However, in the following theorem we are able to obtain the limiting behaviour of $\big(T(t),\;Y(t)\big)$ when both the velocity of the process and the rate of the direction changes are infinite.
\begin{thm}\label{thm:bivdeg}For $\lambda,\, c\to+\infty$, with $\frac{\lambda}{c^2}\to1$, the partial differential equation (\ref{hbarpdeint}) becomes \begin{equation}\frac{\partial u}{\partial t}=-\frac{1}{2}\frac{\partial u}{\partial s}+\frac{1}{4}\;\frac{(1-p)+(1-q)}{(1-p)^2+(1-q)^2}\frac{\partial^2 u}{\partial y^2.}\label{bivdegenerate}\end{equation}\end{thm}\begin{proof}The proof can be performed by straightfoward calculation after dividing equation (\ref{hbarpdeint}) by $\lambda^3$. Simplifying the expression and taking the limit yields the result. \end{proof}
\noindent In order to interpret theorem \ref{thm:bivdeg} observe that, by taking the Fourier transform of both sides of equation (\ref{bivdegenerate}), it can be shown that the characteristic function of $\big(T(t),\;Y(t)\big)$ becomes, in the hydrodynamic limit,
\begin{equation}\lim_{\lambda,c\to+\infty}\mathbb{E}\left[e^{i\left(\alpha T(t)+\beta Y(y)\right)}\right]=e^{i\frac{\alpha t}{2}-\frac{(1-p)+(1-q)}{(1-p)^2+(1-q)^2}\frac{\beta^2t}{4}}.\label{fiumicino}\end{equation} Taking the inverse Fourier transform of the expression (\ref{fiumicino}) yields
\begin{equation}\lim_{\lambda,c\to+\infty}\mathbb{P}\Big(T(t)\in \mathop{ds},\;Y(t)\in\mathop{dy}\Big)/\left(\mathop{ds}\mathop{dy}\right)=\sqrt{\frac{(1-p)^2+(1-q)^2}{(1-p)+(1-q)}}\cdot \frac{e^{-\frac{(1-p)^2+(1-q)^2}{(1-p)+(1-q)}\,\frac{y^2}{t}}}{\sqrt{\pi t}}\cdot \delta\left(s-\frac{t}{2}\right).\label{mort}\end{equation}
Equation (\ref{mort}) implies that, while $Y(t)$ converges to a Brownian motion, the process $T(t)$ becomes deterministic and its limiting value is equal to $\frac{t}{2}$. Thus, in the hydrodynamic limit, the process $\big(X(t),\;Y(t)\big)$ spends half of the time moving vertically.\\

We are now interested in studying the behaviour of $\big(T(t),\;Y(t)\big)$ on the boundary of $R_{ct}$. We start by considering the oblique side with equation $y=cs$ and we define the probability density function \begin{equation}\bar{f}(s,t)=\mathbb{P}\Big(T(t)\in\mathop{ds},\,Y(t)=c\,T(t)\Big)/\mathop{ds},\qquad s\in(0,t).\label{fbar}\end{equation}
By setting, for $j=0,1,2$, \begin{equation*}\bar{f}_j(s,t)=\mathbb{P}\Big(T(t)\in\mathop{ds},\,Y(t)=c\,T(t),\,D(t)=d_j\Big)/\mathop{ds},\qquad s\in(0,t)\end{equation*} we clearly have that
\begin{equation*}
\begin{dcases}
\bar{f}_0(s,t+dt)=\bar{f}_0(s,t)(1-\lambda\mathop{dt})+\bar{f}_1(s,t)\lambda q\mathop{dt}+\bar{f}_2(s,t)\lambda (1-p-q)\mathop{dt}+o(dt)\\
\bar{f}_1(s,t+dt)=\bar{f}_1(s-\mathop{ds},t)(1-\lambda\mathop{dt})+\bar{f}_0(s,t)\lambda p\mathop{dt}+\bar{f}_2(s,t)\lambda q\mathop{dt}+o(dt)\\
\bar{f}_2(s,t+dt)=\bar{f}_2(s,t)(1-\lambda\mathop{dt})+\bar{f}_0(s,t)\lambda (1-p-q)\mathop{dt}+\bar{f}_1(s,t)\lambda p\mathop{dt}+o(dt)
\end{dcases}
\end{equation*}
which implies that
\begin{equation}\label{fbarsys}
\begin{dcases}
\frac{\partial \bar{f}_0}{\partial t}=\lambda q\bar{f}_1+\lambda (1-p-q)\bar{f}_2-\lambda \bar{f}_0\\
\frac{\partial \bar{f}_1}{\partial t}=-\frac{\partial \bar{f}_1}{\partial s}+\lambda p\bar{f}_0+\lambda q\bar{f}_2-\lambda \bar{f}_1\\
\frac{\partial \bar{f}_2}{\partial t}=\lambda (1-p-q)\bar{f}_0+\lambda p\bar{f}_1-\lambda \bar{f}_2
\end{dcases}
\end{equation} with initial conditions $\bar{f}_j(s,0)=\frac{1}{4}\,\delta(s)$ for $j=0,1,2$. Thus, the probability density function (\ref{fbar}) satisfies the third-order partial differential equation

\begin{align}\bigg\{\left(\frac{\partial}{\partial t}+\lambda\right)^3+\left(\frac{\partial}{\partial t}+\lambda\right)^2&\frac{\partial}{\partial s} -\lambda^2\left[2 p q + (1-p-q)^2\right] \left(\frac{\partial}{\partial t}+\lambda\right)\nonumber\\& -\lambda^2(1-p-q)^2\frac{\partial}{\partial s} -\lambda^3(1-p-q)(p^2+q^2)\bigg\}\bar{f}=0.\label{fbarpde}\end{align}
\noindent While finding a general solution to equation (\ref{fbarpde}) is a difficult task, we observe that a noticeable simplification of the equation occurs if the process does not admit reversion, that is if $p+q=1$. In this case equation (\ref{fbarpde}) can be expressed in the form
\begin{equation}\left(\frac{\partial}{\partial t}+\lambda\right)\Bigg(\frac{\partial^2}{\partial t^2}+\left(\frac{\partial}{\partial s}+2\lambda\right)\frac{\partial}{\partial t}+\lambda^2\left[1-2p(1-p)\right]+\lambda\frac{\partial}{\partial s}\Bigg)\bar{f}=0.\label{fbarpdenoref}\end{equation}
Thus, if reflection is not admitted, we are able to study the exact distribution of $\big(T(t),Y(t)\big)$ on the oblique sides of $R_{ct}$. We start by calculating the probability of the particle lying on the upper oblique side.
\begin{thm}\label{thmobliqueprob}If the process $\big(X(t),Y(t)\big)$ does not admit reflection, that is if $p+q=1$, it holds that
\begin{align}\mathbb{P}\big(Y(t)=c\,T(t)\big)=\frac{e^{-\lambda t}}{8}\Bigg\{&\left(1+\frac{1}{\sqrt{2p(1-p)}}\right)^2 e^{\lambda t\sqrt{2p(1-p)}}\nonumber\\
&+\left(1-\frac{1}{\sqrt{2p(1-p)}}\right)^2 e^{-\lambda t\sqrt{2p(1-p)}}-\frac{(2p-1)^2}{p(1-p)}\Bigg\}.\label{obliqueprob}\end{align}
\end{thm}
\begin{proof}We start by studying the probability of $\big(T(t),Y(t)\big)$ lying on the oblique side of $R_{ct}$ conditional on the initial direction of $\big(X(t),Y(t)\big)$ being $d_1$. Clearly, we can write that 
\begin{align}\mathbb{P}\big(Y(t)&=c\,T(t)\big\lvert D(0)=d_1\big)=\sum_{k=0}^{\infty}\mathbb{P}\big(Y(t)=c\,T(t),\; N(t)=k\big\lvert D(0)=d_1\big)\nonumber\\
=&\sum_{k=0}^{\infty}\mathbb{P}\big(Y(t)=c\,T(t),\; N(t)=2k\big\lvert D(0)=d_1\big)+\sum_{k=0}^{\infty}\mathbb{P}\big(Y(t)=c\,T(t),\; N(t)=2k+1\big\lvert D(0)=d_1\big)\nonumber\\
=&\sum_{k=0}^{\infty}\mathbb{P}\big(Y(t)=c\,T(t)\big\lvert D(0)=d_1,\; N(t)=2k\big)\,\mathbb{P}\left(N(t)=2k\right)\nonumber\\\;&\;\;+\sum_{k=0}^{\infty}\mathbb{P}\big(Y(t)=c\,T(t)\big\lvert D(0)=d_1,\; N(t)=2k+1\big)\,\mathbb{P}\left(N(t)=2k+1\right)\nonumber\\
=&e^{-\lambda t}\sum_{k=0}^{\infty}\mathbb{P}\big(Y(t)=c\,T(t)\big\lvert D(0)=d_1,\; N(t)=2k\big)\,\frac{(\lambda t)^{2k}}{(2k)!}\nonumber\\\;&\;\;+e^{-\lambda t}\sum_{k=0}^{\infty}\mathbb{P}\big(Y(t)=c\,T(t)\big\lvert D(0)=d_1,\; N(t)=2k+1\big)\,\frac{(\lambda t)^{2k+1}}{(2k+1)!}.\label{longproofcalc}\end{align}
\noindent We now observe that \begin{equation}\label{longproofcalc1}\mathbb{P}\big(Y(t)=c\,T(t)\big\lvert D(0)=d_1,\; N(t)=2k\big)=\left(2p(1-p)\right)^k.\end{equation}
In order to clarify the formula above, observe that the event $\{Y(t)=c\,T(t)\big\lvert D(0)=d_1,\; N(t)=2k\}$ occurs if the process alternates the upward direction $d_1$ with horizontal directions, which can be both $d_0$ and $d_2$. In order for $N(t)$ to be equal to $2k$, a change of direction from $d_1$ to horizontal and then back to $d_1$ must occur $k$ times, and it always occurs with probability $p(1-p)$ no matter whether the horizontal direction is $d_0$ or $d_2$. Therefore, a path having initial direction $d_1$ and involving $2k$ changes of directions which alternate $d_1$ and horizontal directions can occur with probability $\left(p(1-p)\right)^k$. Moreover, there are $2^k$ equiprobable ways of combining all the possible choices of the horizontal directions of the sequence, since every horizontal direction must belong to the set $\{d_0,\, d_2\}$.\\
\noindent Similar arguments permit to show that \begin{equation}\label{longproofcalc2}\mathbb{P}\big(Y(t)=c\,T(t)\big\lvert D(0)=d_1,\; N(t)=2k+1\big)=\left(2p(1-p)\right)^k.\end{equation}
Formula (\ref{longproofcalc2}) holds because, in order to perform $2k+1$ changes of direction, the process must first perform $2k$ changes of direction with probability $\left(2p(1-p)\right)^k$ and then perform an additional change. Since the initial direction is $d_1$, the first $2k$ changes lead to the direction $d_1$ again. The event $\{Y(t)=c\,T(t)\}$ is then true if and only if the additional change of direction occurs towards a horizontal direction, which occurs with probability 1 because reflection is not admitted by hypothesis.\\
\noindent By now substituting formulas (\ref{longproofcalc1}) and (\ref{longproofcalc2}) into (\ref{longproofcalc}), we obtain
\begin{align}\mathbb{P}\big(Y(t)=c\,T(t)\big\lvert &D(0)=d_1\big)=\sum_{k=0}^{\infty}\mathbb{P}\big(Y(t)=c\,T(t),\; N(t)=k\big\lvert D(0)=d_1\big)\nonumber\\
=&e^{-\lambda t}\sum_{k=0}^{\infty}\frac{(\sqrt{2p(1-p)}\lambda t)^{2k}}{(2k)!}+\frac{e^{-\lambda t}}{\sqrt{2p(1-p)}}\sum_{k=0}^{\infty}\frac{(\sqrt{2p(1-p)}\lambda t)^{2k+1}}{(2k+1)!}\nonumber\\
=&e^{-\lambda t}\Bigg\{\cosh\left(\lambda t \sqrt{2p(1-p)}\right)+\frac{1}{\sqrt{2p(1-p)}}\,\sinh\left(\lambda t \sqrt{2p(1-p)}\right)\Bigg\}\nonumber\\
=&\frac{e^{-\lambda t}}{2}\Bigg\{\left(1+\frac{1}{\sqrt{2p(1-p)}}\right)e^{\lambda t\sqrt{2p(1-p)}}\nonumber\\
&\;\;\;\;\;\;\;+\left(1-\frac{1}{\sqrt{2p(1-p)}}\right)e^{-\lambda t\sqrt{2p(1-p)}}\Bigg\}.\nonumber\end{align}
Similar arguments permit to obtain the probability of $\big(T(t),Y(t)\big)$ lying on the oblique side of $R_{ct}$ conditional on the initial direction being horizontal. In particular, we have that
\begin{align}\mathbb{P}\big(Y(t)=c\,T(t)\big\lvert D(0)=d_0\big)=e^{-\lambda t}\Bigg\{&\frac{p}{2}\left(\frac{1}{2p(1-p)}+\frac{1}{\sqrt{2p(1-p)}}\right)e^{\lambda t\sqrt{2p(1-p)}}\nonumber\\
+&\frac{p}{2}\left(\frac{1}{2p(1-p)}-\frac{1}{\sqrt{2p(1-p)}}\right)e^{-\lambda t\sqrt{2p(1-p)}}+1-\frac{1}{2(1-p)}\Bigg\}\nonumber\end{align}
and
\begin{align}\mathbb{P}\big(Y(t)=c\,T(t)\big\lvert D(0)=d_2\big)=e^{-\lambda t}\Bigg\{&\frac{1-p}{2}\left(\frac{1}{2p(1-p)}+\frac{1}{\sqrt{2p(1-p)}}\right)e^{\lambda t\sqrt{2p(1-p)}}\nonumber\\
+&\frac{1-p}{2}\left(\frac{1}{2p(1-p)}-\frac{1}{\sqrt{2p(1-p)}}\right)e^{-\lambda t\sqrt{2p(1-p)}}+1-\frac{1}{2p}\Bigg\}.\nonumber\end{align}

The theorem is finally proved by observing that $$\mathbb{P}\big(Y(t)=c\,T(t)\big)=\frac{1}{4}\,\sum_{j=0}^2\mathbb{P}\big(Y(t)=c\,T(t)\big\lvert D(0)=d_j\big).$$
\end{proof}

\noindent We are now able to obtain the distribution (\ref{fbar}) in explicit form.
\begin{thm}\label{thm:fbar}If $p+q=1$, the probability density function (\ref{fbar}) is
\begin{align}\label{boundarydistrthm}\bar{f}(s,t)=\frac{\lambda}{2\sqrt{2}}\;&I_0\Big(2\lambda\sqrt{2p(1-p)}\,\sqrt{s(t-s)}\Big)\nonumber\\&+\frac{1}{4}\left(1+\frac{1}{2p(1-p)}\right)\frac{\partial}{\partial t}I_0\Big(2\lambda\sqrt{2p(1-p)}\,\sqrt{s(t-s)}\Big),\qquad\lvert s\in(0,t).\nonumber\end{align}\end{thm}
\begin{proof}In order to find a solution to equation (\ref{fbarpdenoref}), we solve the equation \begin{equation*}\Bigg(\frac{\partial^2}{\partial t^2}+\left(\frac{\partial}{\partial s}+2\lambda\right)\frac{\partial}{\partial t}+\lambda^2\left[1-2p(1-p)\right]+\lambda\frac{\partial}{\partial s}\Bigg)\bar{f}=0.\label{factorfbarpde}\end{equation*} By setting $$\widetilde{\bar{f}}(s,t)=e^{\lambda t}\bar{f}(s,t)$$ the equation is transformed into \begin{equation}\label{p+q=1reduced}\Bigg(\frac{\partial^2}{\partial t^2}+\frac{\partial^2}{\partial s\partial t}-2\lambda^2p(1-p)\Bigg)\widetilde{\bar{f}}=0.\end{equation} The change of variables $z=\sqrt{s(t-s)}$ transforms equation (\ref{p+q=1reduced}) into the Bessel equation \begin{equation*}\label{besselfbar}\frac{d^2\widetilde{\bar{f}}}{z^2}+\frac{1}{z}\frac{d\widetilde{\bar{f}}}{dz}-8\lambda^2p(1-p)\widetilde{\bar{f}}=0\end{equation*} which admits general solution in the form  \begin{equation*}\label{besselfbarsol}\widetilde{\bar{f}}(z)=A\;I_0\Big(2\lambda\sqrt{2p(1-p)}\,z\Big)+B\;K_0\Big(2\lambda\sqrt{2p(1-p)}\,z\Big).\end{equation*} By discarding the second kind modified Bessel function $K_0$ we write that \begin{equation}\label{fbarp+q=1general}\bar{f}(s,t)=e^{-\lambda t}\left[A\;I_0\Big(2\lambda\sqrt{2p(1-p)}\,\sqrt{s(t-s)}\Big)+B\;\frac{\partial}{\partial t}I_0\Big(2\lambda\sqrt{2p(1-p)}\,\sqrt{s(t-s)}\Big)\right].\end{equation} In view of the integral (\ref{s(t-s)besselintegral}), the coefficients $A$ and $B$ which make formula (\ref{fbarp+q=1general}) consistent with formula (\ref{obliqueprob}) are $$A=\frac{\lambda}{2\sqrt{2}},\qquad B=\frac{1}{4}\left(1+\frac{1}{2p(1-p)}\right).$$\end{proof}

\noindent We conclude our study of the distribution (\ref{fbar}) in the case $p+q=1$ by obtaining, in the following theorem, the characteristic function of $\big(T(t),Y(t)\big)$ on the oblique side of $R_{ct}$.

\begin{thm}\label{thmfbarchfp+q=1}If $p+q=1$, the characteristic function of $\big(T(t),Y(t)\big)$ on the oblique side of $R_{ct}$ reads
\begin{align}\label{obliquedistrchf}\mathbb{E}&\left[e^{i\alpha T(t)}\;\mathds{1}_{\{Y(t)=c\, T(t)\}}\right]\nonumber\\
&=\frac{1}{16}\left(\frac{1+2p(1-p)}{p(1-p)}+\frac{4p(1-p)(2\lambda+i\alpha)-i\alpha(1+2p(1-p))}{p(1-p)\sqrt{8\lambda^2p(1-p)-\alpha^2}}\right)e^{-\lambda t+\frac{t}{2}\left(i\alpha+\sqrt{8\lambda^2p(1-p)-\alpha^2}\right)}\nonumber\\
&\;\;+\frac{1}{16}\left(\frac{1+2p(1-p)}{p(1-p)}-\frac{4p(1-p)(2\lambda+i\alpha)-i\alpha(1+2p(1-p))}{p(1-p)\sqrt{8\lambda^2p(1-p)-\alpha^2}}\right)e^{-\lambda t+\frac{t}{2}\left(i\alpha-\sqrt{8\lambda^2p(1-p)-\alpha^2}\right)}\nonumber\\
&\;\;\;\;\;-\frac{(2p-1)^2}{8p(1-p)}\;e^{-\lambda t}.\end{align}
\end{thm}
\begin{proof}We set $$\widehat{\bar{f}}(s,t)=\mathbb{E}\left[e^{i\alpha T(t)}\;\mathds{1}_{\{Y(t)=c\, T(t)\}}\right].$$
By taking into account equation (\ref{fbarpdenoref}), the characteristic function $\widehat{\bar{f}}(s,t)$ can be obtained by finding the solution to the ordinary differential equation
\begin{equation}\label{obliquechfode}
\begin{dcases}
\frac{d^3\widehat{\bar{f}}}{dt^3}+\left[3\lambda-i\alpha\right]\frac{d^2\widehat{\bar{f}}}{dt^2}+\lambda\left[\lambda\left(3-2p(1-p)\right)-2i\alpha\right]\frac{d\widehat{\bar{f}}}{dt}+\lambda^2\left[\lambda\left(1-2p(1-p)\right)-i\alpha\right]\widehat{\bar{f}}=0\\
\widehat{\bar{f}}(\alpha,0)=\frac{3}{4}\\
\left.\frac{d}{d t}\widehat{\bar{f}}(\alpha,t)\right\lvert_{t=0}=\frac{i\alpha-\lambda}{4}\\
\left.\frac{d^2}{d t^2}\widehat{\bar{f}}(\alpha,t)\right\lvert_{t=0}=\frac{2\lambda^2p(1-p)-\alpha^2}{4}.
\end{dcases}\end{equation}
The general solution to equation (\ref{obliquechfode}) reads
$$\widehat{\bar{f}}(\alpha,t)=k_0(\alpha)\;e^{-\lambda t+\frac{t}{2}\left(i\alpha+\sqrt{8\lambda^2p(1-p)-\alpha^2}\right)}+k_1(\alpha)\;e^{-\lambda t+\frac{t}{2}\left(i\alpha-\sqrt{8\lambda^2p(1-p)-\alpha^2}\right)}+k_2(\alpha)\;e^{-\lambda t}.$$
\noindent and the coefficients $k_0,\;k_1$ and $k_2$ can be obtained by using the initial conditions.
\end{proof}

So far, we have studied the distribution of $\big(T(t),Y(t)\big)$ on the oblique side of the triangle $R_{ct}$ and we observed that, while finding the general distribution is a difficult problem, an explicit representation for the distribution can be obtained if reflection is not admitted. We now discuss another special case in which equation (\ref{fbarpde}) can be replaced by a simpler second-order equation, namely the case in which $p=q$. By defining \begin{equation*}\bar{f}_{02}(s,t)=\mathbb{P}\big(T(t)\in\mathop{ds},\,Y(t)=cs,\,D(t)\in\{d_0,\,d_2\}\big)/\mathop{ds},\qquad s\in(0,t)\end{equation*} it is clear that, under the assumption that $p=q$, the system of equations (\ref{fbarsys}) can be written in the form
\begin{equation}\label{fbarsysp=q}
\begin{dcases}
\frac{\partial \bar{f}_{02}}{\partial t}=-2\lambda p\bar{f}_{02}+2\lambda p\bar{f}_{1}\\
\frac{\partial \bar{f}_1}{\partial t}=-\frac{\partial \bar{f}_1}{\partial s}+\lambda p\bar{f}_{02}-\lambda \bar{f}_1
\end{dcases}
\end{equation} which implies that the probability density function (\ref{fbar}) satisfies the second-order partial differential equation
\begin{equation}\label{fbarpdep=q}\left(\frac{\partial^2}{\partial t^2}+\frac{\partial ^2}{\partial s\,\partial t}+\lambda(1+2p)\frac{\partial}{\partial t}+2\lambda p\frac{\partial}{\partial s}+2\lambda^2p(1-p)\right)\bar{f}=0.\end{equation}
In principle, we would start the analysis of the density function $\bar{f}$ by first studying the probability of the particle with position $\big(T(t),Y(t)\big)$ being on the oblique side of $R_{ct}$. Unfortunately, a combinatorial approach in the fashion of theorem \ref{thmobliqueprob} is not trivial if $p=q$ because reversion is admitted, except for the special case $p=q=\frac{1}{2}$. Therefore, for $p=q$, we start by studying the characteristic function in the following theorem.
\begin{thm}If $p=q$, the characteristic function of $\big(T(t),Y(t)\big)$ on the oblique side of $R_{ct}$ reads
\begin{align}\label{obliquedistrchfp=q}\mathbb{E}&\left[e^{i\alpha T(t)}\;\mathds{1}_{\{Y(t)=c\, T(t)\}}\right]\nonumber\\
=&\frac{3}{8}\left(1+\frac{\lambda-i\alpha+6\lambda p}{3\sqrt{\lambda^2(12p^2-4p+1)-\alpha^2-2i\alpha\lambda(1-2p)}}\right)\;e^{\frac{t}{2}\left(i\alpha-\lambda(1+2p)+\sqrt{\lambda^2(12p^2-4p+1)-\alpha^2-2i\alpha\lambda(1-2p)}\right)}\nonumber\\&\;\;+\frac{3}{8}\left(1-\frac{\lambda-i\alpha+6\lambda p}{3\sqrt{\lambda^2(12p^2-4p+1)-\alpha^2-2i\alpha\lambda(1-2p)}}\right)\;e^{\frac{t}{2}\left(i\alpha-\lambda(1+2p)-\sqrt{\lambda^2(12p^2-4p+1)-\alpha^2-2i\alpha\lambda(1-2p)}\right)}.\end{align}
\end{thm}
\begin{proof}Similarly to theorem \ref{thmfbarchfp+q=1}, we set $$\widehat{\bar{f}}(s,t)=\mathbb{E}\left[e^{i\alpha T(t)}\;\mathds{1}_{\{Y(t)=c\, T(t)\}}\right].$$
Since $p=q$, equation (\ref{fbarpdep=q}) implies that the characteristic function $\widehat{\bar{f}}(s,t)$ is the solution to the ordinary differential equation
\begin{equation}\label{obliquechfodep=q}
\begin{dcases}
\frac{d^2\widehat{\bar{f}}}{dt^2}+\left[\lambda\left(1+2p\right)-i\alpha\right]\frac{d\widehat{\bar{f}}}{dt}+\left[2\lambda^2p(1-p)-2i\alpha\lambda p\right]\widehat{\bar{f}}=0\\
\widehat{\bar{f}}(\alpha,0)=\frac{3}{4}\\
\left.\frac{d}{d t}\widehat{\bar{f}}(\alpha,t)\right\lvert_{t=0}=\frac{i\alpha-\lambda}{4}.
\end{dcases}\end{equation}
The general solution to equation (\ref{obliquechfodep=q}) reads
\begin{align}\widehat{\bar{f}}(\alpha,t)=k_0(\alpha)\;&e^{\frac{t}{2}\left(i\alpha-\lambda(1+2p)+\sqrt{\lambda^2(12p^2-4p+1)-\alpha^2-2i\alpha\lambda(1-2p)}\right)}\nonumber\\&\qquad+ k_1(\alpha)\;e^{\frac{t}{2}\left(i\alpha-\lambda(1+2p)-\sqrt{\lambda^2(12p^2-4p+1)-\alpha^2-2i\alpha\lambda(1-2p)}\right)}.\end{align}
\noindent and the coefficients $k_0,\;k_1$ and $k_2$ are obtained by using the initial conditions.
\end{proof}
We are now able to calculate the probability of the process  $\big(T(t),Y(t)\big)$ lying on the oblique side of $R_{ct}$.
\begin{thm}If $p=q$, it holds that
\begin{align}\label{obliqueprobp=q}\mathbb{P}\big(Y(t)=c\,T(t)\big)=&\frac{3}{8}\left(1+\frac{(1+6p)}{3\sqrt{12p^2-4p+1}}\right)\;e^{\frac{\lambda t}{2}\left(-(1+2p)+\sqrt{12p^2-4p+1}\right)}\nonumber\\&\;\;+\frac{3}{8}\left(1-\frac{(1+6p)}{3\sqrt{12p^2-4p+1}}\right)\;e^{\frac{\lambda t}{2}\left(-(1+2p)-\sqrt{12p^2-4p+1}\right)}.\end{align}
\end{thm}\begin{proof}Formula (\ref{obliqueprobp=q}) follows immediately from formula (\ref{obliquedistrchfp=q}) by setting $\alpha=0$.\end{proof}
\noindent It is interesting to observe that formulas (\ref{obliqueprob}) and (\ref{obliqueprobp=q}) coincide if $p=q=\frac{1}{2}$.\\
Unfortunately, the inverse Fourier transform of the characteristic function (\ref{obliquedistrchfp=q}) is difficult to find. Therefore, in the case $p=q$ we are not able to obtain the exact distribution of the process on the oblique side of $R_{ct}$.\\

We conclude our work by studying the distribution of the process $\big(T(t),Y(t)\big)$ on the vertical side of $R_{ct}$. Of course, the particle with position $\big(T(t),Y(t)\big)$ lies on the vertical side if and only if $T(t)=t$, that is if the process $\big(X(t),Y(t)\big)$ has always moved vertically. Recalling formula (\ref{verticalboundary}), we have that \begin{equation}\label{verticalboundary2}\mathbb{P}\left(T(t)=t\right)=\frac{1}{2}\;e^{-\lambda t(p+q)}.\end{equation} We are now interested in studying the probability density function 
 \begin{equation}\bar{g}(y,t)\mathop{dy}=\mathbb{P}\Big(Y(t)\in\mathop{dy},\,T(t)=t\Big),\qquad |y|<ct.\label{gbar}\end{equation}

\noindent Moreover, for $j=1,3$, we define the functions \begin{equation*}\bar{g}_j(y,t)\mathop{dy}=\mathbb{P}\Big(Y(t)\in\mathop{dy},\,T(t)=t,\,D(t)=d_j\Big),\qquad |y|<ct.\end{equation*}
We have that
\begin{equation}\label{gbarsys}
\begin{dcases}
\frac{\partial \bar{g}_1}{\partial t}=-c\frac{\partial \bar{g}_1}{\partial y}+\lambda(1-p-q)\bar{g}_3-\lambda \bar{g}_1\\
\frac{\partial \bar{g}_3}{\partial t}=c\frac{\partial \bar{g}_3}{\partial y}+\lambda(1-p-q)\bar{g}_1-\lambda \bar{g}_3
\end{dcases}
\end{equation}
which implies that the partial differential equation \begin{equation}\label{gbarpde}\left(\frac{\partial^2 }{\partial t^2}+2\lambda\frac{\partial}{\partial t}-c^2\frac{\partial^2}{\partial y^2}+\lambda^2(p+q)(2-p-q)\right)\bar{g}=0.\end{equation}
is satisfied. We now observe that the system of equations (\ref{gbarsys}) and equation (\ref{gbarpde}) perfectly resemble equations (\ref{diagonaldiffeqs}) and (\ref{diagonalpde}), where the variable $x$  is replaced by $y$. Therefore, the motion on the diagonal of the square $S_{ct}$ is identical in distribution to the process moving on the vertical side of the triangle $R_{ct}$. The reason for which this relationship holds is clear. Both the motions take place at costant velocity $c$ on the segment $[-ct,\,ct]$ and the changes of direction occur at Poisson times with constant intensity $\lambda(1-p-q)$, that is when the process $\big(X(t),Y(t)\big)$ performs a reflection. Thus, by using the ideas of theorems \ref{thm:diagdistr} and \ref{thm:diagchf}, the exact distribution of $\big(T(t),Y(t)\big)$ on the vertical side of $R_{ct}$ can be obtained in explicit form. In particular, if $p+q<1$, we have that 
$$\bar{g}(y,t)=\frac{e^{-\lambda t}}{4c}\left[\lambda(1-p-q)\; I_0\left(\frac{\lambda}{c}\,(1-p-q)\,\sqrt{c^2t^2-y^2}\right)+\frac{\partial}{\partial t}I_0\left(\frac{\lambda}{c}\,(1-p-q)\,\sqrt{c^2t^2-y^2}\right)\right]$$
and 
\begin{align}\mathbb{E}\left[e^{i\beta Y(t)}\;\mathds{1}_{\{T(t)=t\}}\right]=\frac{e^{-\lambda t}}{4}&\left[\left(1+\frac{\lambda(1-p-q)}{\sqrt{\lambda^2(1-p-q)^2-\beta^2c^2}}\right)e^{t\sqrt{\lambda^2(1-p-q)^2-\beta^2c^2}}\right.\nonumber\\+&\left.\left(1-\frac{\lambda(1-p-q)}{\sqrt{\lambda^2(1-p-q)^2-\beta^2c^2}}\right)e^{-t\sqrt{\lambda^2(1-p-q)^2-\beta^2c^2}}\right].\nonumber\end{align}

\bibliographystyle{plain}
\nocite{*}
\bibliography{bibliography}
\end{document}